\documentclass[10pt]{article}

\usepackage{amsmath}    % need for subequations
\usepackage{graphicx}   % need for figures
\usepackage{verbatim}   % useful for program listings
\usepackage{color}      % use if color is used in text
\usepackage{subfigure}  % use for side-by-side figures
\usepackage{hyperref}   % use for hypertext links, including those to external documents and URLs
\usepackage{fullpage}
\usepackage{bbm}
\usepackage{dsfont}

\usepackage{amssymb,amsmath}
\usepackage{amsthm}

\newtheorem{theorem}{Theorem}
\newtheorem{proposition}{Proposition}
\newtheorem{lemma}{Lemma}

\newtheorem{corollary}{Corollary}

\newtheorem{conj}{Conjecture}

\newcommand{\Z}{\mathbb{Z}}
\newcommand{\R}{\mathbb{R}}

\newcommand{\sym}{\mathrm{sym}}

\newcommand{\real}{\mathop{\rm Re}}

\newcommand{\sgn}{\mathop{\rm sgn}}

\newcommand{\ord}{\mathop{\rm ord}}

\title{Moments of $L'(1/2)$ in the Family of Quadratic Twists}
\author{Ian Petrow}
\date{}

\begin{document}

\maketitle

\begin{abstract}
We prove the asymptotic formulae for several moments of derivatives of $GL(2)$ $L$-functions over quadratic twists.  The family of $L$-functions we consider has root number fixed to $-1$ and odd orthogonal symmetry.  Assuming GRH we prove the asymptotic formulae for (1) the second moment with one secondary term, (2) the moment of two distinct modular forms $f$ and $g$ and (3) the first moment with controlled weight and level dependence.  We also include some immediate corollaries to elliptic curves via the modularity theorem and the work of Gross and Zagier.
\end{abstract}

\vspace{10pt}

%\section{Introduction}

The values of $L$-functions $L(s,f)$ at certain special half-integral points are of central importance in number theory, c.f. the Birch and Swinnerton-Dyer conjecture.  Analytic methods have been used successfully to study the behavior of these special values in some family of objects, but much remains unknown.  In this paper we study the central values of derivatives of $L$-functions of holomorphic $GL(2)$ modular forms in the family of quadratic twists.  The mean value of this family has been studied successfully in the past by several authors, notably Bump, Friedberg and Hoffstein \cite{BFH}, Murty and Murty \cite{MurtyMurty}, Iwaniec \cite{IwLvanishing} and Munshi \cite{MunshiMoments2}, \cite{MunshiMoments3}.

When $f \otimes \chi_d$ has even functional equation an asymptotic formula for the second moment of $L(1/2,f \otimes \chi_d)$ was computed assuming the generalized Riemann hypothesis (GRH) by Soundararajan and Young \cite{SoundYoung}.  Here, we apply their techniques to several moment problems of comparable difficulty when the sign of the functional equation is $-1$ and the derivative $L'(1/2,f\otimes \chi_d)$ is the correct object of study.  The family of quadratic twists with root number $+1$ as considered by Soundararajan and Young has even orthogonal symmetry in the sense of random matrix theory, while the family we consider has root number $-1$ and odd orthogonal symmetry.  Surprisingly, we find that stronger results are possible in the odd case: the analogues of theorems \ref{thm2} and \ref{thm3} of are out of reach when the root number of $f \otimes \chi_d$ is $1$ and one studies the $L$-functions themselves.  As in Soundararajan and Young, our work is conditional on GRH, but we only use this hypothesis to obtain a useful upper bound to the corresponding un-differentiated moment problem, see conjectures \ref{conjecture1} and \ref{conjecture2}.  The deduction of the necessary upper bounds from GRH is due to Soundararajan \cite{SoundMoments}.  We restrict our attention to holomorphic forms in this paper, but our results should carry over to Maass forms with only minor modifications to the proofs.  

Before stating our results, let us fix some notation and recall some standard facts which can be found in chapter 14 of \cite{IK}.  We consider the space of cuspidal holomorphic modular forms of even weight $\kappa$ on the congruence subgroup $\Gamma_0(N)$ with trivial central character.  Such forms have a Fourier expansion of the form \[f(z)=\sum_{n\geq 1}\lambda_f(n)n^{(\kappa-1)/2}\exp(2 \pi i z).\] We fix a basis of newforms which are eigenfunctions of the Hecke operators and have $\lambda_f(1)=1$.  From now on, we assume all forms $f$ which we work with are elements of this basis.  The Hecke eigenvalues of $f$ are all real (by  the adjointness formula and multiplicity one principle), and hence $f$ is self-dual.  We study the family of twists of $f$ by quadratic characters.  Let $d$ be a fundamental discriminant relatively prime to $N$, and let $\chi_d(\cdot) = \left(\frac{d}{\cdot}\right)$ denote the primitive quadratic character of conductor $|d|$.  Then $f\otimes \chi_d$ is a newform on $\Gamma_0(N|d|^2)$ and the twisted $L$-function is defined for $\real(s)>1$ by \[L(s,f\otimes \chi_d) := \sum_{n \geq 1} \frac{\lambda_f(n)}{n^s} \chi_d(n)= \prod_{p\nmid Nd} \left(1-\frac{\lambda_f(p)\chi_d(p)}{p^s}+\frac{1}{p^{2s}}\right)^{-1} \prod_{p\mid N} \left(1-\frac{\lambda_f(p)\chi_d(p)}{p^s}\right)^{-1}.\]   The completed $L$-function is defined by \[ \Lambda(s,f \otimes \chi_d) := \left(\frac{|d|\sqrt{N}}{2 \pi}\right)^s\Gamma\left(s+\frac{\kappa-1}{2}\right)L(s,f \otimes \chi_d).\] It has the functional equation \[\Lambda(s,f\otimes \chi_d) = i^\kappa \eta \chi_d(-N) \Lambda(1-s, f \otimes \chi_d),\] where $\eta$ is given by the eigenvalue of the Fricke involution, which is independent of $d$ and always $\pm 1$.  We denote the root number by $w(f \otimes \chi_d):= i^\kappa \eta \chi_d(-N)$.  Note that if $d$ is a fundamental discriminant, then $\chi_d(-1) = \pm1$ depending as whether $d$ is positive or negative.   In this paper we work with positive discriminants so that $\chi_d(-N)=\chi_d(N)$, but we could just as easily formulate our results with negative discriminants.  

We are interested here in the derivative of the $L$-function, which also has a Dirichlet series convergent in a right half-plane: \[L'(s,f\otimes \chi_d) = -\sum_{n=1}^\infty\frac{\lambda_f(n)\chi_d(n) \log n}{n^s}.\] It also has a functional equation \[\Lambda'(s, f\otimes \chi_d) = -  i^\kappa \eta \chi_d(-N) \Lambda'(1-s, f \otimes \chi_d)\] with sign opposite to that of $L(s,f\otimes \chi_d)$.  When $w(f\otimes \chi_d) = -1,$ one has that $L(1/2,f\otimes \chi_d)=0$ and $L'(1/2,f \otimes \chi_d)$ is the more appropriate object for study.  

I would like to acknowledge the support of the number theory community at Stanford, and I would especially like to thank Professor Soundararajan for many fruitful discussions.

\section{Statement of Main Results}

In the results of this section we assume the generalized Riemann hypothesis (GRH) for the zeta function, the family of quadratic twists of $f$ and $g$ and the symmetric square of $f$ and $g$.  See also the comments immediately before and after conjectures \ref{conjecture1} and \ref{conjecture2}, below.  We use the notations $(d,\square)=1$ or $ \mathcal{D}$ to denote the sets of square-free integers or fundamental discriminants, respectively.  Let $F:\mathbb{R}_{\geq 0}\rightarrow \mathbb{R}_{\geq 0}$ be a fixed smooth function with compact support closely resembling the indicator function of the interval $[0,1]$, and let $\widetilde{F}(s)=\int_0^\infty F(x)x^{s-1}\,dx$ denote its Mellin transform.  We formulate our results for the subset of $\mathcal{D}$ of integers which are $4$ times a $2 \mod 4$ squarefree integer, but could have just as well picked out the other congruence classes which together constitute $\mathcal{D}$.  The subscripts on the symbols $O$ and $\ll$ indicate that the implied constants depend only on those parameters.  
\begin{theorem}\label{thm1}  Assume GRH, and let $F(\cdot)$ be a smooth approximation to the indicator function of $[0,1]$ with compact support.  For any  normalized cuspidal Hecke newform $f$ with trivial central character, odd level $N$ and even weight $\kappa$ we have \begin{equation*}\begin{split} \sum_{\substack{(d,2 N\square)=1 \\ w(f \otimes \chi_{8d})=-1}} L'(1/2,f\otimes \chi_{8d})^2F(8d/X) = \frac{X}{\pi^2}  L(1, \sym^2 f)^3 Z^*(0,0)\widetilde{F}(1) \left(\frac{1}{3} \log^3 X  + C_2(f) \log^2 X\right) \\  + O_{\kappa,N,\varepsilon}\left(X (\log X)^{1+\varepsilon}\right).\end{split}\end{equation*} In the above \[C_2(f) = \frac{\Gamma'(\kappa/2)}{\Gamma(\kappa/2)}+ \log \frac{\sqrt{N}}{2\pi} + \gamma +3\frac{L'(1,\sym^2 f)}{L(1,\sym^2 f)}+\frac{\frac{d}{du}Z^*(u,0)\vert_{u=0}}{Z^*(0,0)} + \frac{\widetilde{F}'(1)}{\widetilde{F}(1)},\]  $\gamma$ is Euler's constant, and $Z^*(u,v)$ is a holomorphic function defined by \eqref{prod1} and \eqref{Z1} for $\real(u),\real(v) > -1/4+\varepsilon$ given by a sum of two absolutely convergent Euler products and is uniformly bounded in $u,v$ where it converges. Moreover, $Z^*(0,0) =0$ if and only if the root number $w(f) = 1$ and $N$ is square, in which case the moment vanishes identically.  \end{theorem}

By the celebrated theorem of Gross and Zagier \cite{GrossZagier}, theorem \ref{thm1} also gives the variance of canonical heights of Heegner points on an elliptic curve associated with $f$.  Note that the analogue of theorem \ref{thm1} without the derivative is the main result of Soundararajan and Young \cite{SoundYoung}.  In this paper, we compute the main terms in a slightly different manner than do Soundararajan and Young, and applying our technique to the second moment without derivatives improves the error term there to $\ll_{\kappa,\varepsilon} X(\log X)^{1/2+\varepsilon}$.  Nonetheless, shifted moments are still crucial to the theorem of Soundararajan and Young, whereas they are not necessary here.

The next theorem is a moment for two distinct modular forms $f$ and $g$.  Theorem \ref{thm2} is particularly interesting because the asymptotic formula for the analogous moment without derivatives is completely out of reach by current techniques. \begin{theorem}\label{thm2}Assume GRH, and let $F(\cdot)$ be a smooth approximation to the indicator function of $[0,1]$ with compact support.  For any two distinct normalized cuspidal Hecke newforms $f$ and $g$ with trivial central characters, odd levels $N_1$ and $N_2$, and even weights $\kappa_1$ and $\kappa_2$ we have
\[\sum_{\substack{(d,2 N_1N_2\square)=1 \\ w(f \otimes \chi_{8d})=-1 \\ w(g \otimes \chi_{8d})=-1}}  L'(1/2,f\otimes \chi_{8d})L'(1/2,g \otimes \chi_{8d})F(8d/X) = C(f,g)X \log^2 X + O_{ f,g, \varepsilon}\left(X (\log X)^{1+\varepsilon}\right) .\] In the above \[ C(f,g) = \frac{1 }{2\pi^2}  L(1,\sym^2 f)L(1,\sym^2 g)L(1,f \otimes g)Z^*(0,0)\widetilde{F}(1),\] where $Z^*(u,v)$ is a holomorphic function defined by \eqref{prod2} and \eqref{Z2} in $\real(u),\real(v) \geq -1/4+\varepsilon$, depending on $f$ and $g$, given by a sum of four absolutely convergent Euler products and uniformly bounded in $u,v$ where it converges.  Moreover, $Z^*(0,0) =0$ if and only if either the root number $w(f) = 1$ and $N_1$ is square or the root number $w(g)=1$ and $N_2$ is square. In either of these two cases the moment vanishes identically. \end{theorem}

Lastly, theorem \ref{thm3} below is a first moment in the twist aspect with controlled dependence on both the weight $\kappa$ and level $N$.  Again, the analogue of theorem \ref{thm3} without the derivative is completely out of reach, but would have interesting corollaries, see \cite{LiuYoung}.  \begin{theorem}\label{thm3} Assume GRH, and let $F(\cdot)$ be a smooth approximation to the indicator function of $[0,1]$ with compact support.  For any $A>0$ and any  normalized cuspidal Hecke newform $f$ with trivial central character, odd level $N$ and even weight $\kappa$ we have \begin{equation*}\begin{split}  \sum_{\substack{(d,2 N\square)=1 \\ w(f \otimes \chi_{8d})=-1}} L'(1/2,f\otimes \chi_{8d})F(8d/X)  =   C_3(f)X \left(\log \frac{X\kappa \sqrt{N}}{2\pi} + 2\frac{L'(1,\sym^2 f)}{L(1,\sym^2 f)}+ \frac{Z^{* '}(0)}{Z^*(0)} \right) \\ + O_{A,\varepsilon}\left( X  (\log X \kappa N)^{1/4+\varepsilon} +  \frac{X^{13/17} (\kappa N)^{4/17} }{ (\log X\kappa N)^{A}} \right), \end{split}\end{equation*} In the above \[ C_3(f) = \frac{\widetilde{F}(1) }{2 \pi^2} L(1, \sym^2 f) Z^*(0)\] and $Z^*(u)$ is a holomorphic function defined by \eqref{prod3} and \eqref{Z3} as a sum of two absolutely convergent Euler products for $\real(u)>-1/4+\varepsilon$.  Moreover, $Z^*(0)=0$ if and only if the root number $w(f)=1$ and $N$ is a square.  If so, then the moment vanishes identically, and if not \[  Z^*(0) \gg \frac{\log \log N}{(\log N)^{1/2}}, \] uniformly in $\kappa$.   \end{theorem}  

Thus our methods break convexity in the dependence on $\kappa $ and $N$ in the error term by an arbitrary power of log.  Using GRH once again, we obtain non-vanishing results.  By applying the technique from \cite{IK} theorem 5.17 we have that \[ \frac{L'(1,\sym^2 f)}{L(1,\sym^2 f)} + \frac{Z^{* '}(0)}{Z^*(0)}  \ll \log \log \kappa N.\] These terms therefore may be subsumed into the error term in theorem \ref{thm3}.  In the same vein, by theorem 5.19 of \cite{IK} one has the bound \[L(1, \sym^2 f) \gg (\log \log \kappa N)^{-1}.\]  From these estimates and theorem \ref{thm3} the following corollary is obtained. \begin{corollary}\label{nonvanishingcor}
Assume GRH.  If the root number of $f$ is $1$ then assume also that the level of $f$ is not an integer square. For any $A>0$ there exists an odd squarefree $d$ relatively prime to $N$ with $d \ll_A \kappa N / (\log \kappa N)^A $ for which \[w(f \otimes \chi_{8d})=-1 \,\,\,\,\,\,\,\,\,\,\, \text{ and } \,\,\,\,\,\,\,\,\,\, L'(1/2, f\otimes \chi_{8d}) > 0.\] \end{corollary} 

If $E/\mathbb{Q}$ is an elliptic curve given by the Weierstauss equation $y^2=f(x),$ we may define the twisted elliptic curve $E^d/\mathbb{Q}$ by the equation $dy^2=f(x)$.  By the work of Gross and Zagier \cite{GrossZagier} and the modularity theorem \cite{BCDT} we have the following corollary.  \begin{corollary}\label{smallrank1twists} Assume GRH.   Let $E/\mathbb{Q}$ be an elliptic curve of odd conductor $N$.  If the root number of $E$ is $1$, then assume also that the conductor $N$ is not an integer square.  For every $A>0$ there exist odd squarefree $d$ relatively prime to $N$ with $d \ll_A N/(\log N)^A$ for which the curve $E^{8d}/\mathbb{Q}$ has root number $-1$ and Mordell-Weil rank exactly 1.  \end{corollary}

One expects the convexity bound here to be a non-vanishing twist of size $d \ll_\varepsilon (\kappa N)^{1+\varepsilon},$ see e.g. Hoffstein and Kontorovich \cite{HoffsteinKontorovich}.  Our non-vanishing corollaries on GRH are, in fact, quite weak.  As previously remarked by many authors, the method of moments is an inefficient way to produce non-vanishing theorems.  If one is willing to assume GRH, the methods of Iwaniec, Luo and Sarnak \cite{ILS}, \"{O}zl\"{u}k and Snyder \cite{OS2}, \cite{OS3} or Heath-Brown \cite{HBAvgRank} adapted to small nonvanishing twists should yield better results.  We postpone carrying out this line of research to a future paper, and moreover, we believe that the theorems \ref{thm1}, \ref{thm2} and \ref{thm3} have interest independent of the corollaries.  

We do not use the full strength of GRH in theorems \ref{thm1}, \ref{thm2} or \ref{thm3}.  In fact, in the case of the first two all we need is the following conjecture. \begin{conj}\label{conjecture1} Let $\varepsilon>0$, and $t$ be a real number with $|t| \leq X$ and $1/2 \leq \sigma \leq 1/2+1/\log X.$  Then \[\sum_{\substack{d \in \mathcal{D} \\(d,N)=1 \\  |d|\leq X}} |L(\sigma+it,f\otimes \chi_d)|^2 \ll_{\kappa,N,\varepsilon} X(\log X)^{1+\varepsilon}.\] \end{conj} Theorem \ref{thm3} on the other hand is true if we assume than $N$ is odd squarefree and conjecture \ref{conjecture2} in place of GRH. \begin{conj}\label{conjecture2} Let $\varepsilon>0$, and $t$ be a real number with $|t| \leq X$ and $1/2 \leq \sigma \leq 1/2+1/\log X.$  Then \[\sum_{\substack{d \in \mathcal{D} \\(d,N)=1\\ |d| \leq X}} |L(\sigma+ it, f \otimes \chi_d)| \ll_{\varepsilon} X (\log X\kappa N )^{1/4+\varepsilon}.\] \end{conj}  The work of Soundararajan \cite{SoundMoments} shows that conjecture \ref{conjecture1} follows from the GRH for the Riemann zeta function, the family of quadratic twists of $f$ and the symmetric square of $f$.  By keeping track of the dependence on $\kappa$ and $N$ in Soundararajan's proof, one finds that the GRH for quadratic twists of $f$, the Riemann zeta function, and the symmetric square of $f$ implies conjecture \ref{conjecture2}.  Unconditionally, all that is known towards conjectures \ref{conjecture1} and \ref{conjecture2} is a bound of the form $\ll_{f,\varepsilon} (X(1+|t|))^{1+\varepsilon}$ due to Heath-Brown's quadratic large sieve \cite{HBQuadSieve}.  It seems that obtaining the results of this paper unconditionally should not be completely out of reach, but nonetheless, doing so requires additional ideas.

Let us briefly describe the main difficulties in proving the above theorems, some previous attacks on these difficulties, and the new input in our work which allows us to overcome them.  

Take for example theorem \ref{thm1}.  After applying the approximate functional equation and pulling the sum over $d$ inside one encounters a sum of the form \[\sum_d \chi_d(n_1n_2) F\left(\frac{d}{U}\right) \] for some cut-off function $F$, where $\chi_d$ is the quadratic character modulo $d$.  One wants to apply Poisson summation to this sum, but the length of the sum $U\approx X$ is comparable to the square root of the conductor $\sqrt{n_1n_2}$, so the dual sum that one obtains is of the same shape as the original.  This is the familiar ``deadlock'' situation described, for example, in the paper of Munshi \cite{MunshiMoments2}, or by multiple Dirichlet series, for example in \cite{DGH}.  This deadlock has been broken in some ways before.  Soundararajan and Young find that the second moment of $L(1/2,f\otimes \chi_d)$ is transformed by Poisson summation to the dual problem of finding an estimate of the integral over shifts $it_1$ and $it_2$ of the same moment.  They exploit this transformation using GRH to obtain upper bounds on shifted moments to prove their theorem.  Munshi observes in the paper \cite{MunshiMoments2} that taking derivatives amplifies the main term of moments but does not affect the error term.  He uses this fact to unconditionally obtain an asymptotic formula for the first moment of higher derivatives $\Lambda^{(\ell)}(1/2,f \otimes \chi_d)$ with $\ell \geq 8$ weighted by the number of representations of $d$ as a sum of two squares (a situation with conductor of similar length to ours).  Munshi also solves a similar problem in \cite{MunshiMoments3} obtaining an asymptotic for the first derivative in the special case that $f$ corresponds to a CM elliptic curve.  

In our paper, we observe that taking a derivative concentrates the mass of $L'(1/2,f\otimes\chi_d)$ in the terms of the approximate functional equation with small $n$.  When we truncate $U \leq X/(\log X)^{100}$ we gain something from Poisson summation, and treat the tail separately.  The idea behind bounding the tail is that \[L'(1/2,f\otimes \chi_d) \approx \sum_{n\leq |d|} \frac{\lambda_f(n)\chi_d(n)\log\frac{|d|}{n}}{n^{1/2}},\] so that when $|d|/(\log |d|)^{100} \leq n \leq |d|$ we have that the $0\leq \log |d|/n \ll \log \log |d|$ are quite small.  These terms look essentially like the series for $L(1/2,f \otimes \chi_d)$, the moments of which are smaller than moments of the derivative.  We are then able to use Soundararajan's upper bounds assuming GRH \cite{SoundMoments} to bound the tail.  The idea is that the dual sum of a moment of $L'(1/2,f\otimes \chi_d)$ looks like a moment of the un-differentiated $L(1/2,f\otimes \chi_d),$ which we exploit to obtain our results.  

\section{Approximate Functional Equation}

We begin with a lemma which will be used in all three theorems.
\begin{lemma}[Approximate functional equation]\label{approxfe}
Let $f$ be a $\lambda_f(1)=1$ normalized cuspidal newform on $\Gamma_0(N)$ with trivial central character and root number $w(f)=i^\kappa \eta$.  Let $Z>0$ be an arbitrary real number parameter.  Define the cut-off function \[W_Z(x) := \frac{1}{2 \pi i} \int_{(3)} \frac{\Gamma(u+\kappa/2)}{\Gamma(\kappa/2)} \left(\frac{2 \pi x}{Z\sqrt{N}}\right)^{-u}\frac{1-u\log Z}{u^2}\,du.\]  Then  \[ \sum_{n \geq 1} \frac{\lambda_f(n) \chi_d(n)}{n^{1/2}} W_Z\left(\frac{n}{|d|}\right) -i^\kappa\eta \chi_d(-N)\sum_{n \geq 1} \frac{\lambda_f(n) \chi_d(n)}{n^{1/2}} W_{Z^{-1}}\left(\frac{n}{|d|}\right)= \begin{cases} L'(1/2,f \otimes \chi_d) & \text{if } w(f \otimes \chi_{d}) = -1 \\ 0 & \text{if } w(f \otimes \chi_d)=1. \end{cases}  \]
\end{lemma}

\begin{proof}
We follow Iwaniec and Kowalski \cite{IK} Section 5.2. Take \begin{eqnarray*} I(Z,f,s) & := & \frac{1}{2 \pi i } \int_{(3)} \Lambda(s+u,f \otimes \chi_d)Z^{u} \frac{1-u\log Z}{u^2}\,du \\ &= & \Lambda'(s,f\otimes \chi_d) + \frac{1}{2 \pi i} \int_{(-3)} \Lambda(s+u,f \otimes \chi_d) Z^{u}\frac{1-u\log Z}{u^2}\,du, \end{eqnarray*}  so that by a change of variables and an application of the functional equation we have \[ I(Z,f,s)=\Lambda'(s,f\otimes \chi_d) + i^\kappa \eta \chi_d(-N)I(Z^{-1},f,1-s). \]   If the root number $w(f \otimes \chi_d)=-1$, we take $s=1/2$ to find \[ L'(1/2, f \otimes \chi_d) = \sum_{n\geq 1}\frac{\lambda_f(n)\chi_d(n)}{n^{1/2}}W_Z\left(\frac{n}{|d|}\right)+\chi_d(-N)\sum_{n\geq 1}\frac{\lambda_f(n)\chi_d(n)}{n^{1/2}}W_{Z^{-1}}\left(\frac{n}{|d|}\right),\] as in the statement of the lemma.  On the other hand, if the root number of $f\otimes \chi_d$ is 1, then $\Lambda'(1/2, f \otimes \chi_d)=0$, hence \[I(Z,f,1/2)-i^\kappa \eta \chi_d(-N)I(Z^{-1},f,1/2) = 0.\]  Thus the lemma holds for both cases of root number of $f \otimes \chi_d$.    \end{proof}  In the proof of theorems \ref{thm1} and \ref{thm2} we will use $Z=1$ so that the approximate functional equation takes a particularly simple form.  Let $N$ be the level of $f$.  In the proof of theorem \ref{thm3} we take $Z=N^{1/2}$ to compensate for the asymmetry in estimates in level aspect introduced from averaging over root numbers.   Note that the only difference in the approximate functional equation for $L'(1/2,f\otimes \chi_d)$ as opposed to that of $L(1/2,f \otimes \chi_d)$ is the sign of the root number, and the denominator of the integrand of $W(x),$ which becomes $u^2$ instead of $u$.  Therefore, many of the calculations necessary for our results are identical to those in the paper of Soundararajan and Young \cite{SoundYoung}.

\section{Proof of Theorem \ref{thm1}}\label{ProofofTheorem1}

We prove theorem \ref{thm1} by splitting the sums in the approximate functional equation (lemma \ref{approxfe}), and using proposition \ref{prop1} below to compute the main terms.  

\begin{proof}[Proof of Theorem \ref{thm1}.]  Let $F$ be a smooth, nonnegative, compactly supported function on $\R_{>0}$, and recall the definition of $W(x)=W_1(x)$ from the approximate functional equation (lemma \ref{approxfe}).  For a parameter $U \leq X/ (\log X)^{100}$ define the truncated sum \[\mathcal{A}_U(1/2,f \otimes \chi_{8d}) := (1-i^\kappa \eta \chi_d(-N)) \sum_{n=1}^\infty \frac{\lambda_f(n)\chi_{8d}(n)}{\sqrt{n}} W\left(\frac{n}{U}\right), \]  and define the tail $\mathcal{B}_U(1/2,f \otimes \chi_{8d})$ by setting $L'(1/2,f \otimes \chi_{8d}) = \mathcal{A}_U(1/2,f \otimes \chi_{8d})+\mathcal{B}_U(1/2,f \otimes \chi_{8d}).$  Define the sums \begin{eqnarray*} \mathrm{I}_U(f) & := & \sum_{(d,2N\square)=1} L'(1/2,f \otimes \chi_{8d}) \mathcal{A}_U(1/2,f \otimes \chi_{8d}) F(8d/X) \\ \mathrm{II}_U(f) & :=  & \sum_{(d,2N\square)=1} \mathcal{A}_U(1/2,f \otimes \chi_{8d})^2 F(8d/X) \\ \mathrm{III}_U(f) & := & \sum_{(d,2N\square)=1} \mathcal{B}_U(1/2,f \otimes \chi_{8d})^2 F(8d/X). \end{eqnarray*} so that we have the decomposition \[ \sum_{(d,2N\square)=1}L'(1/2,f \otimes \chi_{8d})^2  F(8d/X) = 2 \mathrm{I}_U(f) - \mathrm{II}_U(f) + \mathrm{III}_U(f).\]  Using the below proposition \ref{prop1} we will be able to give asymptotic formulae for $\mathrm{I}_U(f)$ and $\mathrm{II}_U(f)$, and using conjecture \ref{conjecture1} we will obtain an upper bound on $\mathrm{III}_U(f)$ smaller than the main terms.  Applying this decomposition in Soundararajan and Young's work improves the error term there to $O\left(X(\log X)^{1/2+\varepsilon}\right)$.  

For $N' =1$ or $N$, and $h(x,y,z)$ some smooth cut-off function let \[S(N',h) := \sum_{(d,2N\square)=1} \sum_{n_1=1}^\infty\sum_{n_2=1}^\infty \frac{\lambda_f(n_1)\lambda_f(n_2)}{\sqrt{n_1n_2}}\chi_{8d}(N'n_1n_2)h(d,n_1,n_2).\]  \begin{proposition}\label{prop1} Assume GRH or conjecture \ref{conjecture1}.  Let $X,U_1,U_2$ large, $U_1U_2\leq X^2$, and $N$ odd.  Let $h(x,y,z)$ be a smooth function on $\R_{>0}^3,$ with compact support in $x$, having all partial derivatives extending continuously to the boundary, satisfying \[x^iy^jz^kh^{(i,j,k)}(x,y,z) \ll_{i,j,k} \left(1+\frac{x}{X}\right)^{-100}  \left(\log \frac{U_1}{y} \right)\left(1+\frac{y}{U_1}\right)^{-100}\left( \log \frac{U_2}{z}\right)  \left(1+\frac{z}{U_2}\right)^{-100}.\] Set $h_1(y,z) = \int_0^\infty h(xX,y,z)\,dx.$  Then \[S(N',h) = \frac{4X}{\pi^2} \sum_{\substack{(n_1n_2,2)=1 \\ N'n_1n_2 = \square}} \frac{\lambda_f(n_1) \lambda_f(n_2)}{\sqrt{n_1n_2}}\prod_{p|Nn_1n_2} \frac{p}{p+1} h_1(n_1,n_2) + O_{\kappa, N}\left( ( U_1U_2)^{1/4} X^{1/2} (\log X)^{11}\right).\]  \end{proposition} This proposition and its proof are nearly identical to the main proposition from the paper of Soundararajan and Young \cite{SoundYoung} (see proposition 3.1 and the remarks in \S 5 of that paper) except for minor details of generalizing from full level to arbitrary level $N$, so we omit the proof.  The main idea is to use Poisson summation (see lemma \ref{PoissonSummation}) to evaluate the sum over discriminants $d$, and conjecture \ref{conjecture1} to bound the dual sum thereby obtained.  

We now proceed to the computation of $\mathrm{I}_U(f)$ and $\mathrm{II}_U(f)$.  Let $h(x,y,z) = F(8x/X)W(y/U)W(z/8x).$  In the notation of proposition \ref{prop1} we have by the approximate functional equation that \[\mathrm{I}_U(f) = 2S(1,h)-2i^\kappa \eta S(N,h).\]  For notational ease, set $G(u) := \Gamma(\kappa/2+u)\Gamma(\kappa/2)^{-1}(\sqrt{N}/2 \pi)^{u}$ which, recall, appears in the function $W(x)$.  Let $\widetilde{F}(v) = \int_0^\infty F(x) x^{v-1}\,dx$ denote the Mellin transform and set \[Z_{N'}(u,v) = \sum_{\substack{(n_1n_2,2)=1 \\ N'n_1n_2=\square}} \frac{\lambda_f(n_1)\lambda_f(n_2)}{n_1^{1/2+u} n_2^{1/2+v}} \prod_{p|Nn_1n_2}\frac{p}{p+1}.\]  Applying proposition \ref{prop1} and Mellin inversion, we find that \begin{equation}\label{integral1}\mathrm{I}_U(f) = \frac{X}{\pi^2} \frac{1}{(2 \pi i )^2} \int_{(1)}\int_{(1)} \frac{G(u)G(v)}{u^2v^2} U^uX^v \widetilde{F}(1+v)\left(Z_N(u,v)-i^\kappa \eta Z_1(u,v)\right)\,du\,dv+O_{\kappa,N}(X).\end{equation}

We compute for either $N'= 1$ or $N$ that $Z_{N'}(u,v)$ has the Euler product  \begin{equation*} \begin{split} Z_{N'}(u,v)=\prod_{p \nmid 2N} \left( 1+\frac{p}{p+1} \left[ \frac{1}{2}\left(1-\frac{\lambda_f(p)}{p^{1/2+u}} +\frac{1}{p^{1+2u}}\right)^{-1}\left(1-\frac{\lambda_f(p)}{p^{1/2+v}}+\frac{1}{p^{1+2v}}\right)^{-1} \right. \right. \\ \left. \left.+  \frac{1}{2}\left(1+\frac{\lambda_f(p)}{p^{1/2+u}}+\frac{1}{p^{1+2u}}\right)^{-1}\left(1+\frac{\lambda_f(p)}{p^{1/2+v}}+\frac{1}{p^{1+2v}}\right)^{-1} -1\right]\right) \end{split} \end{equation*}\begin{equation}\label{prod1} \begin{split}\times \prod_{p|N} \frac{p}{p+1} \left[ \frac{1}{2}\left(1-\frac{\lambda_f(p)}{p^{1/2+u}} \right)^{-1}\left(1-\frac{\lambda_f(p)}{p^{1/2+v}}\right)^{-1} + (-1)^{\ord_p(N')}  \frac{1}{2}\left(1+\frac{\lambda_f(p)}{p^{1/2+u}}\right)^{-1}\left(1+\frac{\lambda_f(p)}{p^{1/2+v}}\right)^{-1} \right] .\end{split} \end{equation}  Hence we have that \[Z_{N'}(u,v) = \zeta(1+u+v)L(1+2u, \sym^2 f)L(1+u+v, \sym^2 f)L(1+2v, \sym^2 f) Z_{N'}^*(u,v),\] where $Z_N^*(u,v)$ and $Z_1^*(u,v)$ are given by some absolutely convergent Euler products and are uniformly bounded in the region $\real(u), \real(v) \geq -1/4+\varepsilon$ in $u,v,\kappa$ and $N$.  Set $Z(u,v) := Z_N(u,v)-i^\kappa \eta Z_1(u,v)$ and \begin{equation}\label{Z1}Z^*(u,v) :=Z_N^*(u,v)-i^\kappa \eta Z_1^*(u,v).\end{equation}  A careful inspection of \eqref{Z1} and \eqref{prod1}, using positivity of $(1\pm \lambda_f(p)p^{-1/2})^{-1}$ shows that $Z^*(0,0)=0$ if and only if $\varepsilon(f)=1$ and $N$ is a square.  

We now compute by shifting contours of \eqref{integral1}.  Start the lines of integration at $\real(u)=\real(v) = 1/10$, and begin the computation with shifting the $v$ integration to the $\real(v) = -1/5$ line.  We encounter poles at $v=0$ and $v=-u$.  The remaining double integral on the lines $\real(v)=-1/5$ and $\real(u) = 1/10$ is $\ll_{\kappa,N,\varepsilon} X^{-1/10+\varepsilon}$, and the contribution from the simple pole at $v=-u$ is $\ll_{\kappa,N} 1$.  The main term comes from double pole at $v = 0$, giving \begin{equation*}\begin{split} I_U = \frac{ X  }{\pi^2} \widetilde{F}(1)  \frac{1}{2 \pi i} \int_{(1/10)} \frac{G(u)}{u^2} U^{u} Z(u,0)\left( \log X  +G'(0) +\frac{\widetilde{F}'(1)}{\widetilde{F}(1)}  + \frac{ \frac{d}{dv}Z(u,v)\vert_{v=0}}{Z(u,0)}\right)\,du + O_{\kappa,N}(X).\end{split}\end{equation*} Now $Z(u,0)$ has a single pole and $\frac{d}{dv}Z(u,v)\vert_{v=0}$ has a double pole.  Combine these with $u^2$ in the denominator, and we encounter a triple and quadruple pole.  The residue of the triple pole of \[ \frac{G(u)}{u^2}U^u Z(u,0)\] at $u=0$ is given by \[L(1,\sym^2 f)^3Z^*(0,0)\left(\frac{1}{2}\log^2 U + \left[ \frac{\Gamma'(\kappa/2)}{\Gamma(\kappa/2)} + \log \frac{\sqrt{N}}{2 \pi}+\gamma+3\frac{L'(1,\sym^2 f)}{L(1,\sym^2 f)}+ \frac{\frac{d}{du}Z^*(u,0)\vert_{u=0}}{Z^*(0,0)}\right] \log U +O_{\kappa,N}(1)\right)   .\]  The residue of the quadruple pole of \[ \frac{G(u)}{u^2}U^u \frac{d}{dv}Z(u,v)\vert_{v=0} \] at $u=0$ is given by \[ -L(1,\sym^2 f)Z^*(0,0)\left(\frac{1}{6} \log^3 U + \frac{1}{2} \left[\frac{\Gamma'(\kappa/2)}{\Gamma(\kappa/2)}+ \log \frac{\sqrt{N}}{2\pi}\right] \log^2 U + O_{\kappa,N}(\log U)\right).  \]  By shifting the line of integration to $\real(u) = -1/5$, we find that the the remaining integral is $\ll_{\kappa,N,\varepsilon} X^{-1/5+\varepsilon},$ hence collecting the above terms coming from residues, we find that \begin{equation*}\begin{split} \mathrm{I}_U(f) = \frac{X}{\pi^2} L(1,\sym^2 f)^3 Z^*(0,0)\widetilde{F}(1) \left(\frac{1}{2}\log X (\log U)^2 -\frac{1}{6} \log^3 U  + \left[ \frac{\Gamma'(\kappa/2)}{\Gamma(\kappa/2)}+ \log \frac{\sqrt{N}}{2\pi}+\gamma+3\frac{L'(1,\sym^2 f)}{L(1,\sym^2 f)}\right. \right. \\ \left. \left.+ \frac{\frac{d}{du}Z^*(u,0)\vert_{u=0}}{Z^*(0,0)}\right] \log X \log U  +  \frac{1}{2} \frac{\widetilde{F}'(1)}{\widetilde{F}(1)} (\log U)^2 +O_{\kappa,N}(\log X)\right). \end{split}\end{equation*}

The sum $\mathrm{II}_U(f)$ is computed similarly, but with a different choice of $h(x,y,z)$.  As above, the main term comes from the intersection of the two polar divisors $u=0$ and $v=0$.  One finds \begin{equation*}\begin{split} \mathrm{II}_U(f)  =   \frac{ X}{\pi^2} L(1,\sym^2 f)^3Z^*(0,0) \widetilde{F}(1)\left(\frac{1}{3}\log^3U + \left[ \frac{\Gamma'(\kappa/2)}{\Gamma(\kappa/2)}+ \log \frac{\sqrt{N}}{2\pi} + \gamma +3\frac{L'(1,\sym^2 f)}{L(1, \sym^2 f)} \right. \right. \\  \left. \left.+\frac{\frac{d}{du}Z^*(u,0)\vert_{u=0}}{Z^*(0,0)}\right]\log^2 U + O_{\kappa,N}(\log U) \right).\end{split}\end{equation*}

We now give an upper bound for the sum $\mathrm{III}_U(f)$ which, recall, involves $\mathcal{B}_U$.  We have\[\mathcal{B}_U(1/2,8d) = \left(1-i^\kappa \eta \chi_{8d}(N)\right)\frac{1}{2 \pi i} \int_{(2)} \frac{G(s)}{s} L(1/2+s,f \otimes \chi_{8d}) \left(\frac{(8d)^s-U^s}{s}\right)\,ds.\]  Recall that $L(1/2+s,f \otimes \chi_{8d})$ has root number $-1$ and vanishes at $s=0$, therefore the integrand is actually entire and we move the line of integration to the $\real(s) = 1/\log X$ line.   On this line \[\left| \frac{(8d)^s - U^s}{s}\right| \ll \log \left(8d/U\right),\] uniformly in $s$, thus \[\mathcal{B}_U(1/2,8d) \ll |\log 8d/U | \int_{-\infty}^\infty \frac{\left|G\left(\frac{1}{\log X}+it\right)\right|}{\left|\frac{1}{\log X}+it\right|} \left|L\left(\frac{1}{2}+ \frac{1}{\log X}+it,f \otimes \chi_{8d}\right)\right|\,dt.\]  Inserting this in $\mathrm{III}_U(f)$ we have that \begin{equation}\label{Bsum}\begin{split} \mathrm{III}_U(f) \ll (\log X/U)^2\int_{-\infty}^\infty \int_{-\infty}^\infty \frac{\left| G\left(\frac{1}{\log X}+it_1\right)G\left(\frac{1}{\log X}+it_2\right)\right|}{\left|\left(\frac{1}{\log X}+it_1\right)\left(\frac{1}{\log X}+it_2\right)\right|} \\ \times  \sum_{\substack{(d,2N\square)=1 \\ 0<8d \leq X}}\left|L\left(\frac{1}{2}+ \frac{1}{\log X}+it_1,f \otimes \chi_{8d}\right)L\left(\frac{1}{2}+ \frac{1}{\log X}+it_2,f \otimes \chi_{8d}\right)\right| \,dt_1\,dt_2.\end{split}\end{equation} Use Cauchy-Schwarz to split the sum over $d$ above in two, so that it suffices to bound \[ \int_{-\infty}^\infty \frac{\left| G\left(\frac{1}{\log X}+it\right)\right|}{\left|\left(\frac{1}{\log X}+it\right)\right|} \left( \sum_{\substack{(d,2N\square)=1 \\ 0<8d \leq X}}\left|L\left(\frac{1}{2}+ \frac{1}{\log X}+it,f \otimes \chi_{8d}\right)\right|^2\right)^{1/2} \,dt. \] We have that  \[\int_{-\infty}^\infty \frac{\left| G\left(\frac{1}{\log X}+it\right)\right|^\frac{1}{2}}{\left|\frac{1}{\log X}+it\right|}\,dt \ll \log \log X,\] and \[\left| G\left(\frac{1}{\log X}+ it\right) \right| \sum_{\substack{ (d,2N\square)=1 \\ 0<8d\leq X}} \left|L\left(\frac{1}{2}+ \frac{1}{\log X}+it,f \otimes \chi_{8d}\right)\right|^2 \ll_{f,\varepsilon} X \left(\log X\right)^{1+\varepsilon} \] uniformly in $t$ by conjecture \ref{conjecture1} and the sharp cut-off in $\left| G(1/\log X+it ) \right|$ for large $t$.  Bringing these estimates together we find that \[ \mathrm{III}_U(f) \ll_{\kappa,N,\varepsilon} X (\log X)^{1+\varepsilon} (\log X/U)^2.\]  Note that in contrast to the work of Soundararajan and Young, shifted moments are not necessary to prove our theorem.  

Finally, set $U = X/(\log X)^{100}$.  Note that \[ \left( \log X - \frac{2}{3} \log U\right) \log^2 U = \frac{1}{3} \log^3 X + O_{\varepsilon} ( (\log X)^{1+\varepsilon}),\]   so that pulling together our evaluations of $\mathrm{I}_U(f),\mathrm{II}_U(f)$ and $\mathrm{III}_U(f)$ we find  \begin{equation*} \begin{split} \sum_{(d,2N\square)=1} L'(1/2,f\otimes \chi_{8d})^2F(8d/X) =   \frac{X}{\pi^2}  L(1,\sym^2 f)^3 Z^*(0,0)\widetilde{F}(1) \left(\frac{1}{3} \log^3 X + \left[\frac{\Gamma'(\kappa/2)}{\Gamma(\kappa/2)}+ \log \frac{\sqrt{N}}{2\pi} + \gamma   \right. \right. \\ \left .\left.  +3\frac{L'(1,\sym^2 f)}{L(1,\sym^2 f)}+\frac{\frac{d}{du}Z^*(u,0)\vert_{u=0}}{Z^*(0,0)} + \frac{\widetilde{F}'(1)}{\widetilde{F}(1)}\right] \log^2 X + O_{\kappa,N,\varepsilon}(X (\log X)^{1+\varepsilon}) \right).\end{split}\end{equation*} \end{proof}

\section{Proof of Theorem \ref{thm2}}\label{ProofofTheorem2}

We turn to the moment for two different forms $f$ and $g$ of levels $N_1$ and $N_2$ respectively.  Set $N=N_1N_2$.  The proof of theorem \ref{thm2} is a slight variation on the proof of theorem \ref{thm1}.  \begin{proof}[Proof of theorem \ref{thm2}.] Assume GRH or conjecture \ref{conjecture1}, and that $U \leq X / (\log X)^{100}$.  We split the sum $L'(1/2,f\otimes \chi_{8d}) = \mathcal{A}_U(1/2,f \otimes \chi_{8d}) + \mathcal{B}_U(1/2,f \otimes \chi_{8d}),$ where $\mathcal{A}_U(1/2, f\otimes \chi_{8d})$ and $\mathcal{B}_U(1/2,f \otimes \chi_{8d})$ are defined at the outset of Section \ref{ProofofTheorem1}.  Take the decomposition \begin{equation}\label{fandg}\begin{split}  L'(1/2,f \otimes \chi_{8d})L'(1/2,g \otimes \chi_{8d})  =  L'(1/2,f \otimes \chi_{8d})\mathcal{A}_U(1/2,g \otimes \chi_{8d}) + \mathcal{A}_U(1/2,f \otimes \chi_{8d})L'(1/2,g \otimes \chi_{8d})  \\   - \mathcal{A}_U(1/2,f \otimes \chi_{8d})\mathcal{A}_U(1/2,g \otimes \chi_{8d}) + \mathcal{B}_U(1/2,f \otimes \chi_{8d})\mathcal{B}_U(1/2,g \otimes \chi_{8d}). \end{split}\end{equation} 
Summing over $(d,2N\square)=1$, we have the 4 sums which we denote by \[\mathrm{I}_U(f,g):=\sum_{(d,2N\square)=1} L'(1/2,f \otimes \chi_{8d})\mathcal{A}_U(1/2,g \otimes \chi_{8d}) F(8d/X) ,\] \[\mathrm{I}_U(g,f):=\sum_{(d,2N\square)=1} \mathcal{A}_U(1/2,f \otimes \chi_{8d})L'(1/2,g \otimes \chi_{8d})  F(8d/X) ,\] \[\mathrm{II}_U(f,g):=\sum_{(d,2N\square)=1} \mathcal{A}_U(1/2,f \otimes \chi_{8d})\mathcal{A}_U(1/2,g \otimes \chi_{8d})  F(8d/X) ,\] and \[\mathrm{III}_U(f,g):=\sum_{(d,2N\square)=1} \mathcal{B}_U(1/2,f \otimes \chi_{8d})\mathcal{B}_U(1/2,g \otimes \chi_{8d}) F(8d/X) ,\] so that \[\sum_{(d,2N\square)=1} L'(1/2,f\otimes \chi_{8d})L'(1/2,g \otimes \chi_{8d})F(8d/X) = \mathrm{I}_U(f,g) + \mathrm{I}_U(g,f) - \mathrm{II}_U(f,g) + \mathrm{III}_U(f,g).\]  We can compute precise asymptotic estimates for $\mathrm{I}_U(f,g),$ $\mathrm{I}_U(g,f)$ and $\mathrm{II}_U(f,g)$, meanwhile $\mathrm{III}_U(f,g)$ can be reduced by Cauchy-Schwarz to the sum $\mathrm{III}_U(f)$ from the proof of theorem \ref{thm1}.  Hence \[\mathrm{III}_U(f,g) \ll_{\kappa,N,\varepsilon} X (\log X)^{1+\varepsilon}.\]

We next state the proposition which allows us to compute the sums $\mathrm{I}_U(f,g),$ $\mathrm{I}_U(g,f)$ and $\mathrm{II}_U(f,g)$.  Let $N'$ be one of the four choices $N' = 1, N_1,N_2,$ or $N$.  Define \[S_{f,g}(N',h) := \sum_{(d,2N\square)=1} \sum_{n_1=1}^\infty\sum_{n_2=1}^\infty \frac{\lambda_f(n_1)\lambda_g(n_2)}{\sqrt{n_1n_2}}\chi_{8d}(N'n_1n_2)h(d,n_1,n_2).\] \begin{proposition}\label{prop2} Assume GRH or conjecture \ref{conjecture1}.  Let $X,U_1,U_2$ large, $U_1U_2\leq X^2$, and $N =N_1N_2$ odd.  Let $h(x,y,z)$ be a smooth function on $\R_{>0}^3,$ with compact support in $x$, having all partial derivatives extending continuously to the boundary, satisfying \[x^iy^jz^kh^{(i,j,k)}(x,y,z) \ll_{i,j,k}  \left(1+\frac{x}{X}\right)^{-100}\left(\log \frac{U_1}{y}\right) \left(1+\frac{y}{U_1}\right)^{-100} \left( \log \frac{U_2}{z}\right) \left(1+\frac{z}{U_2}\right)^{-100}.\] Set $h_1(y,z) = \int_0^\infty h(xX,y,z)\,dx.$ Then \[S_{f,g}(N',h) = \frac{4X}{\pi^2} \sum_{\substack{(n_1n_2,2)=1 \\ N'n_1n_2 = \square}} \frac{\lambda_f(n_1) \lambda_g(n_2)}{\sqrt{n_1n_2}}\prod_{p|Nn_1n_2} \frac{p}{p+1} h_1(n_1,n_2) + O_{f,g}\left( ( U_1U_2)^{1/4} X^{1/2} (\log X)^{11}\right).\]  \end{proposition}

Proposition \ref{prop2} is a slight variation on proposition \ref{prop1}, so we omit the proof.  The reader should take note of the remarks following proposition \ref{prop1}, as they apply just as well to proposition \ref{prop2}.

Now we proceed to use this proposition to evaluate $\mathrm{I}_U(f,g),$ $\mathrm{I}_U(g,f)$ and $\mathrm{II}_U(f,g)$.  Take for example the case $\mathrm{I}_U(f,g)$, for which we set $h(d,n_1,n_2) = F(8d/X)W(n_1/U)W(n_2/8d).$  By the approximate functional equation (lemma \ref{approxfe}) with $Z=1$ we have that \begin{equation}\label{thm2intermsofS}\begin{split} \sum_{(d,2N\square)=1} L'(1/2,f\otimes \chi_{8d})\mathcal{A}_U(1/2,g\otimes \chi_d)F(8d/X) \hspace{2in} \\ = S_{f,g}(1,h) -i^{\kappa_1}\eta_f S_{f,g}(N_1,h) -i^{\kappa_2}\eta_g S_{f,g}(N_2,h)+i^{\kappa_1+\kappa_2}\eta_f\eta_gS_{f,g}(N,h).\end{split}\end{equation}  Likewise, $\mathrm{I}_U(g,f)$ and $\mathrm{II}_U(f,g)$ are evaluated the same way with $h(d,n_1,n_2) = F(8d/X)W(n_1/8d)W(n_2/U)$ and $h(d,n_1,n_2) = F(8d/X)W(n_1/U)W(n_2/U),$ respectively.  

Next, we evaluate the main terms of the various $S_{f,g}$ in \eqref{thm2intermsofS} by contour integration.  We set $G_f(u) := \Gamma(\kappa_1/2+u)\Gamma(\kappa_1/2)^{-1}(\sqrt{N_1}/2 \pi)^{u}$ to be the Mellin transform of $W_1(x),$ and similarly for $G_g$.  For $N' =1,N_1,N_2$ or $N$ define the Dirichlet series $Z_{N'}(u,v)$ by \[Z_{N'}(u,v) := \sum_{\substack{(n_1n_2,2)=1 \\ N'n_1n_2=\square}} \frac{\lambda_f(n_1)\lambda_f(n_2)}{n_1^{1/2+u} n_2^{1/2+v}} \prod_{p|Nn_1n_2}\frac{p}{p+1}.\] One has therefore that \begin{equation}\label{integral2} S_{f,g}(N',h) =  \frac{X}{2\pi^2} \frac{1}{(2 \pi i )^2} \int_{(1)}\int_{(1)} \frac{G_g(u)G_f(v)}{u^2v^2} U^uX^v \widetilde{F}(1+v)Z_{N'}(u,v)\,du\,dv+O_{\kappa,N}(X).  \end{equation}  Let $\chi_{0,N_i}$ be the trivial Dirichlet character mod $N_i$ for $i=1,2$, that is to say, \[\chi_{0,N_i}(p) = \begin{cases} 1 & \text{ if } p \nmid N_i \\ 0 & \text{ if } p \mid N_i. \end{cases} \]  Then the Euler product for $Z_{N'}(u,v)$ is given by \begin{equation*} \begin{split} Z_{N'}(u,v)= \prod_{p \nmid 2N} \left( 1+\frac{p}{p+1} \left[ \frac{1}{2}\left(1-\frac{\lambda_f(p)}{p^{1/2+u}} +\frac{1}{p^{1+2u}}\right)^{-1}\left(1-\frac{\lambda_g(p)}{p^{1/2+v}}+\frac{1}{p^{1+2v}}\right)^{-1} \right. \right. \\ \left. \left.+  \frac{1}{2}\left(1+\frac{\lambda_f(p)}{p^{1/2+u}}+\frac{1}{p^{1+2u}}\right)^{-1}\left(1+\frac{\lambda_g(p)}{p^{1/2+v}}+\frac{1}{p^{1+2v}}\right)^{-1} -1\right]\right) \end{split} \end{equation*}\begin{equation}\label{prod2} \begin{split}\times \prod_{p|N} \frac{p}{p+1} \left[ \frac{1}{2}\left(1-\frac{\lambda_f(p)}{p^{1/2+u}} +\frac{\chi_{0,N_1}(p)}{p^{1+2u}}\right)^{-1}\left(1-\frac{\lambda_g(p)}{p^{1/2+v}}+\frac{\chi_{0,N_2}(p)}{p^{1+2v}}\right)^{-1} \right.  \\  \left.+ (-1)^{\ord_p(N')}  \frac{1}{2}\left(1+\frac{\lambda_f(p)}{p^{1/2+u}}+\frac{\chi_{0,N_1}(p)}{p^{1+2u}}\right)^{-1}\left(1+\frac{\lambda_g(p)}{p^{1/2+v}}+\frac{\chi_{0,N_2}(p)}{p^{1+2v}}\right)^{-1} \right]. \end{split} \end{equation} 

If $\alpha_f(p)$ and $\beta_f(p)$ are the local roots of $f$ with $\alpha_f(p)+\beta_f(p)=\lambda_f(p)$, then we define for $\real(s)>1$ \[L(s,f\otimes g) = \prod_p \left(1-\frac{\alpha_f(p)\alpha_g(p)}{p^{s}}\right)^{-1}\left(1-\frac{\alpha_f(p)\beta_g(p)}{p^{s}}\right)^{-1}\left(1-\frac{\beta_f(p)\alpha_g(p)}{p^{s}}\right)^{-1}\left(1-\frac{\beta_f(p)\beta_g(p)}{p^{s}}\right)^{-1}, \] and for $\real(s)\leq 1$ by analytic continuation.  Then in any of the four cases $N'=1,N_1,N_2,$ or $N,$ we have that \[Z_{N'}(u,v) = L(1+u+v,f \otimes g) L(1+2u, \sym^2 f)L(1+2v, \sym^2 g) Z^*_{N'}(u,v),\] where $Z_{N'}^*(u,v)$ is given by some absolutely convergent Euler product which is uniformly bounded in the region $\real(u),\real(v) \geq -1/4+\varepsilon$.  Set $Z(u,v) = Z_1(u,v) -i^{\kappa_1}\eta_f Z_{N_1}(u,v) -i^{\kappa_2}\eta_g Z_{N_2}(u,v) + Z_{N}(u,v),$ and \begin{equation}\label{Z2}Z^*(u,v) = Z^*_1(u,v) -i^{\kappa_1} \eta_f Z^*_{N_1}(u,v) -i^{\kappa_2}\eta_g Z^*_{N_2}(u,v) + i^{\kappa_1+\kappa_2}\eta_f\eta_gZ^*_{N}(u,v).\end{equation}  A careful inspection of \eqref{Z2} and \eqref{prod2} using positivity of $(1\pm \lambda_f(p)p^{-1/2})^{-1}$ shows that $Z^*(0,0)=0$ if and only if either root number $w(f)$ or $w(g) =1$, and the corresponding $N_1$ or $N_2$ is a square.  

With this information about $Z_{N'}(u,v)$, one shifts contours of \eqref{integral2} as in the proof of theorem \ref{thm1} to compute the various $S_{f,g}$.  We find that \begin{eqnarray*} \mathrm{I}_U(f,g)  &=& \sum_{(d,2N\square)=1}  L'(1/2,f \otimes \chi_{8d})\mathcal{A}_U(1/2,g \otimes \chi_{8d}) F(8d/X) \\  & = & \frac{ X }{2\pi^2} L(1,\sym^2 f)L(1,\sym^2 g)L(1,f \otimes g)Z^*(0,0) \widetilde{F}(1)  \log X  \log U  + O_{f,g}(X \log X),\end{eqnarray*} and similarly for $\mathrm{I}_U(g,f)$.  We also compute \begin{eqnarray*} \mathrm{II}_U(f,g) &= &  \sum_{(d,2N\square)=1} \mathcal{A}_U(1/2,f \otimes \chi_{8d})\mathcal{A}_U(1/2,g \otimes \chi_{8d}) \\ & = & \frac{ X}{2\pi^2}  L(1,\sym^2 f)L(1,\sym^2 g)L(1,f \otimes g)Z^*(0,0)\widetilde{F}(1) \log^2 U + O_{f,g}(X\log U).\end{eqnarray*}  Finally, setting $U = X/(\log X)^{100}$ we obtain \begin{eqnarray*}& & \sum_{(d,2N\square)=1} L'(1/2,f \otimes \chi_{8d})L'(1/2,g \otimes \chi_{8d}) F(8d/X) \\ & = & \frac{ X}{2\pi^2}  L(1,\sym^2 f)L(1,\sym^2 g)L(1,f \otimes g)Z^*(0,0)\widetilde{F}(1) \log^2 X   + O_{f,g,\varepsilon} \left( X (\log X)^{1+\varepsilon} \right) .\end{eqnarray*}   \end{proof}

\section{Proof of Theorem \ref{thm3}}\label{ProofofTheorem3}

In this section, we apply the techniques of the previous two sections to the first moment of $L'(1/2,f \otimes \chi_{8d})$ over twists, keeping careful track of the dependence on both the weight $\kappa$ and the level $N$.  

\begin{proof}[Proof of Theorem \ref{thm3}.] We prove the theorem by splitting the sum into a main part and tail, and use the asymmetric approximate functional equation (lemma \ref{approxfe}) with $Z=N^{1/2}$.  Assume GRH or conjecture \ref{conjecture2}, and that both $\kappa N \leq X$ and $U \leq X / (\log X\kappa N)^{\frac{17}{4}(A+6)}$ for $A>0$ fixed.  Define the main part \[\mathcal{A}_U(1/2,f\otimes \chi_{8d}) = \sum_{n\geq 1} \frac{\lambda_f(n)\chi_{8d}(n)}{n^{1/2}}W_Z\left(\frac{n}{U}\right)-i^\kappa \eta \chi_{8d}(N)\sum_{n \geq 1} \frac{\lambda_f(n)\chi_{8d}(n)}{n^{1/2}}W_{Z^{-1}}\left(\frac{n}{U}\right),\] and the tail $\mathcal{B}_U(1/2,f \otimes \chi_{8d})=L'(1/2,f\otimes \chi_{8d})-\mathcal{A}_U(1/2,f\otimes \chi_{8d})$ as in Section \ref{ProofofTheorem1}.    Following Soundararajan and Young again, we give the analogue of propositions \ref{prop1} and \ref{prop2} for the first moment.  Let $N'=1$ or $N$, and for $h(x,y)$ a smooth function on $\R_{>0}^2$ set \[ T(N',h) := \sum_{(d,2N\square)=1} \sum_{n=1}^\infty \frac{\lambda_f(n)}{\sqrt{n}} \chi_{8d}(N'n) h(d,n).\]  We will use the following proposition with $z$ equal to either $Z=N^{1/2}$ when $N'=1$, or $Z^{-1}=N^{-1/2}$ when $N'=N$.  \begin{proposition}\label{FirstMomentCalculation} Assume GRH or conjecture 2.  Let $z>0$ be a parameter (c.f. the asymmetric approximate functional equation), and let $X$ and $U$ be large.  Suppose that $N$ is odd, and that  $U\kappa \sqrt{N}z \leq X^2$.  Let $h(x,y)$ be a smooth function on $\R_{>0}^2$ which is compactly supported in $x$, having all partial derivatives extending continuously to the boundary, and satisfying the partial derivative bounds \[x^iy^j h^{(i,j)}(x,y) \ll_{i,j} \left(1+\frac{x}{X}\right)^{-100} \left(\log \frac{U\kappa N}{y}\right) \left(1+\frac{y}{U\kappa \sqrt{N}z}\right)^{-100} .\]  Then, setting $h_1(y) :=\int_0^\infty h(xX,y)\,dx,$ we have \[T(1,h) = \frac{4X}{\pi^2} \sum_{\substack{(n,2)=1 \\ n=\square }} \frac{\lambda_f(n)}{\sqrt{n}} \prod_{p|Nn} \frac{p}{p+1}h_1(n) + O\left( X^{9/17} (U \kappa \sqrt{N}z)^{4/17} (\log X\kappa N)^6\right),\]  \[T(N,h) = \frac{4X}{\pi^2} \sum_{\substack{(n,2)=1 \\ Nn=\square }} \frac{\lambda_f(n)}{\sqrt{n}} \prod_{p|Nn} \frac{p}{p+1}h_1(n) + O\left( X^{1/2} (U \kappa N^{3/2}z)^{1/4} (\log X\kappa N)^6\right).\] \end{proposition}

Proposition \ref{FirstMomentCalculation} is sufficiently different from proposition 3.1 of Soundararajan and Young that we give a detailed proof in Section \ref{pfprop3}. 

Let $h_L(x,y):= F(8x/X)W_{N^{1/2}}(y/U)$ and $h_S(x,y) := F(8x/X)W_{N^{-1/2}}(y/U)$ for ``long'' and ``short'', respectively.  Recall for fundamental discriminants $d>0$ that $\chi_d(-N) = \chi_d(N)$, so that we have in the notation of proposition \ref{FirstMomentCalculation} that the main part of the moment is \[\sum_{(d,2N\square)=1} \mathcal{A}_U(1/2,f\otimes \chi_{8d})F(8d/X) = T(1,h_L)-i^\kappa \eta T(N,h_S).\]  Recalling that $U\kappa \sqrt{N}z \leq X^2$ and taking $z=N^{1/2}$ or $N^{-1/2}$ in proposition \ref{FirstMomentCalculation} we have that \begin{equation}\label{Asum}\begin{split}\sum_{(d,2N\square)=1} \mathcal{A}_U(1/2,f\otimes \chi_{8d})F(8d/X)  = \frac{X}{2\pi^2}\widetilde{F}(1) \sum_{\substack{(n,2)=1 \\ n=\square}} \frac{\lambda_f(n)}{n^{1/2}} \prod_{p\mid Nn}\frac{p}{p+1} W_{N^{1/2}}\left(\frac{n}{U}\right) \hspace{1in} \\ + \frac{X}{2\pi^2}\widetilde{F}(1) \sum_{\substack{(n,2)=1 \\ Nn=\square}} \frac{\lambda_f(n)}{n^{1/2}} \prod_{p\mid Nn}\frac{p}{p+1} W_{N^{-1/2}}\left(\frac{n}{U}\right) + O_A\left(\frac{X^{13/17} (\kappa N)^{4/17}}{(\log X\kappa N)^A}\right).\end{split} \end{equation}

For $N'=1$ or $N$ define the Dirichlet series \[Z_{N'}(u) := \sum_{\substack{(n,2)=1 \\ N' n = \square}} \frac{\lambda_f(n)}{n^{1/2+u}} \prod_{p|Nn} \frac{p}{p+1}.\]  We compute from the definition of $W_Z(x)$ that \begin{equation}\label{MT1}\frac{X}{2 \pi} \widetilde{F}(1) \sum_{\substack{(n,2)=1 \\ n=\square}} \frac{\lambda_f(n)}{\sqrt{n}} \prod_{p|Nn} \frac{p}{p+1} W_{N^{1/2}}\left(\frac{n}{U}\right) = \frac{X}{2\pi^2} \widetilde{F}(1) \frac{1}{2\pi i}\int_{(3)} \frac{\Gamma(u+\kappa/2)}{\Gamma(\kappa/2)}\left(\frac{2 \pi}{UN}\right)^{-u} Z_{1}(u)\frac{1-\frac{1}{2}u\log N}{u^2}\,du ,\end{equation} and in the same way that \begin{equation}\label{MT2}\frac{X}{2 \pi} \widetilde{F}(1) \sum_{\substack{(n,2)=1 \\ Nn=\square}} \frac{\lambda_f(n)}{\sqrt{n}} \prod_{p|Nn} \frac{p}{p+1} W_{N^{-1/2}}\left(\frac{n}{U}\right) = \frac{X}{2\pi^2} \widetilde{F}(1) \frac{1}{2\pi i}\int_{(3)} \frac{\Gamma(u+\kappa/2)}{\Gamma(\kappa/2)}\left(\frac{2 \pi}{U}\right)^{-u} Z_{N}(u) \frac{1+\frac{1}{2}u\log N}{u^2}\,du.\end{equation}   The Dirichlet series $Z_{N'}(u)$ also has an Euler product \begin{equation}\label{prod3} \begin{split} Z_{N'}(u) = \prod_{p\nmid 2N} 1+\frac{p}{p+1}\left[\frac{1}{2}\left(1-\frac{\lambda_f(p)}{p^{1/2+u}}+\frac{1}{p^{1+2u}}\right)^{-1} + \frac{1}{2} \left(1+\frac{\lambda_f(p)}{p^{1/2+u}}+\frac{1}{p^{1+2u}}\right)^{-1} -1\right] \\ \times \prod_{p \mid N} \frac{p}{p+1} \left[\frac{1}{2}\left(1-\frac{\lambda_f(p)}{p^{1/2+u}}\right)^{-1} +(-1)^{\ord_p(N')} \frac{1}{2}\left(1+\frac{\lambda_f(p)}{p^{1/2+u}}\right)^{-1}\right]. \end{split}\end{equation}  We have then that \[Z_{N'}(u) = L(1+2u,\sym^2 f)Z_{N'}^*(u),\] where $Z_{N'}^*(u)$ is given by some absolutely convergent Euler product in the region $\real(u)>-1/4$.  Moreover, inspecting the above Euler product, we see that \[ \frac{1}{\log \log N}\ll Z_1^*(0) \ll \log \log N\] and \[N^{-(1/2+\varepsilon)} \ll_\varepsilon Z_N^*(0) \ll \log \log N,\] uniformly in $\kappa$.  

With this information about $Z_{N'}(u)$, we shift the contours in \eqref{MT1} and \eqref{MT2} to $\real(u) = -4/17$, and pick up the residue from the double pole at $u=0$.  The double pole in \eqref{MT1} or \eqref{MT2} contributes \[\frac{X}{2 \pi^2} \widetilde{F}(1) L(1,\sym^2 f) Z_{N'}^*(0) \left(\log \frac{U\kappa \sqrt{N}}{2\pi} + \frac{Z'_{N'}(0)}{Z_{N'}(0)}+O(\kappa^{-1})\right). \]  We must also bound the integrals \begin{equation}\label{MT1s}  \frac{X}{2\pi^2} \widetilde{F}(1) \frac{1}{2\pi i}\int_{(-4/17)} \frac{\Gamma(u+\kappa/2)}{\Gamma(\kappa/2)}\left(\frac{2 \pi}{UN}\right)^{-u}L(1+2u,\sym^2 f)Z_{1}^*(u)\frac{1-\frac{1}{2}u\log N}{u^2}\,du \end{equation} and \begin{equation}\label{MT2s} \frac{X}{2\pi^2} \widetilde{F}(1) \frac{1}{2\pi i}\int_{(-4/17)} \frac{\Gamma(u+\kappa/2)}{\Gamma(\kappa/2)}\left(\frac{2 \pi}{U}\right)^{-u} L(1+2u,\sym^2 f)Z_{N}^*(u) \frac{1+\frac{1}{2}u\log N}{u^2}\,du.\end{equation} These two are treated a little differently.  Let us begin with the simpler case of \eqref{MT1s}.  We have the convexity bound \begin{equation}\label{convexity} L(9/17+it,\sym^2 f) \ll (\kappa^2 N^2 (1+|t|)^4)^{4/17} (\log \kappa N)^2.\end{equation} by estimating with the approximate functional equation of the symmetric square $L$-function, and the Deligne bound \cite{DeWC} for its coefficients (see for example, equation (5.22) of \cite{IK}).   Hence, the integral \eqref{MT1s} is \[\ll_A X^{13/17}(\kappa N)^{4/17}/(\log X \kappa N)^A.\]  The integral \eqref{MT2s} is a little more delicate, and we need to use the decay of $Z^*_N(u)$ with respect to $N$.  When $\real(u)>-1/4$, we have that \[Z_N^*(u) = \prod_{p \nmid N}\left(1+O(p^{-(2+4u)})\right)\prod_{\substack{p \mid N \\ \ord_p(N) \text{ odd } }} \frac{\lambda_f(p)}{p^{1/2+u}}\left(1+O(p^{-(1+2u)})\right)\prod_{\substack{p \mid N \\ \ord_p(N) \text{ even } }} \left(1+O(p^{-(1+2u)})\right),\] so that \[Z^*_N(u) \ll \left(\prod_{\ord_p(N) \text{ odd }}p\right)^{-1/2-\real(u)}(\log N)^2.\] Assuming e.g. that $N$ is squarefree, this shows that for fixed $\real(u)$ $Z_{N}(u)$ decays as a function of $N$.  If one is willing to assume Lindel\"{o}f, it is unnecessary to use the decay of $Z_N^*(u)$ with respect to $N$, and hence the restriction to squarefree $N$ may be omitted.  Using this along with the convexity bound \eqref{convexity} for $L(1+2u,\sym^2 f)$, we find that \eqref{MT2s} is \[\ll_A X^{13/17}\kappa^{4/17}N^{7/34}/(\log X \kappa N)^A\ll_A X^{13/17}(\kappa N)^{4/17}/(\log X \kappa N)^A,\] so that these integrals are subsumed into the error term in the theorem.  

Now set \begin{equation}\label{Z3}Z^*(u) = Z_1^*(u)-i^\kappa \eta Z_N^*(u)\end{equation} so that we have from \eqref{Asum} that \begin{equation*}\begin{split} \sum_{(d,2N\square)=1} \mathcal{A}_U(1/2,f\otimes \chi_{8d})F(8d/X) = \frac{X}{2 \pi^2} \widetilde{F}(1) L(1,\sym^2 f) Z^*(0) \left(\log \frac{U\kappa \sqrt{N}}{2\pi} + 2\frac{L'(1,\sym^2 f)}{L(1,\sym^2 f)}+ \frac{{Z^*}'(0)}{Z^{*}(0)} \right)\\ + O_A\left( \frac{X^{13/17} (\kappa N)^{4/17}}{(\log X \kappa N)^A} \right). \end{split} \end{equation*}  By carefully inspecting \eqref{prod3} and using that \[\prod_{p \mid N} \left(1+\frac{2}{\sqrt{p}}\right) \ll \frac{(\log N)^{1/2}}{\log \log N} \] we find that $Z^*(0)=0$ if and only if $i^\kappa \eta := w(f) = 1$ and $N$ is a square and that if $Z^*(0) \neq 0$, it is $\gg \log \log N / (\log N)^{1/2},$ uniformly in $\kappa$.  

Now consider the tail \[\sum_{(d,2N\square)=1}\mathcal{B}_U(1/2,f\otimes \chi_{8d})F(8d/X).\]  Recall the notation $G(u) = \Gamma(\kappa/2+u)\Gamma(\kappa/2)^{-1}(\sqrt{N}/2\pi)^u$ from the definition of $W_Z(x)$, and that we have set $Z=N^{1/2}$.  We have in similar fashion to the two preceding theorems that \[\mathcal{B}_U(1/2, 8d)  = \frac{1}{2 \pi i } \int_{(2)} \frac{G(s)}{s}  L(1/2+s,f \otimes \chi_{8d}) \frac{(8d)^s-U^s}{s}\left(Z^s(1-s\log Z) -i^\kappa \eta \chi_{8d}(N)Z^{-s}(1+s \log Z)\right)\,ds.\] The integrand is entire, and we may shift the contour to the line $\real(s)=1/\log X\kappa N$.  On this line we have \[\frac{(8d)^s-U^s}{s}\left(Z^s(1-s\log Z) -i^\kappa \eta \chi_{8d}(N)Z^{-s}(1+s \log Z)\right) \ll \log X/U\] so that \begin{equation*}\begin{split} \sum_{(d,2N\square)=1} \mathcal{B}_U(1/2, 8d) F(8d/X) \ll \left| \log X/U \right| \int_{-\infty}^\infty \frac{\left|  G\left(\frac{1}{\log X\kappa N} + it\right) \right|}{\left| \frac{1}{\log X\kappa N}+it\right|} \hspace{1in} \\  \times  \sum_{\substack{(d,2N\square)=1\\ 0<8d\leq X}}\left| L\left(\frac{1}{2}+\frac{1}{\log X\kappa N} +it, f \otimes \chi_{8d}\right) \right| \,dt .\end{split}\end{equation*} Set $U= X/(\log X\kappa N)^{\frac{17}{4}(A+6)}$.  Using conjecture \ref{conjecture2} when $t$ is small and the cut-off in $|G(1/\log X\kappa N)+it|^{1/2}$ when $t$ is large we have \[  |G(1/\log X\kappa N)+it|^{1/2}\sum_{\substack{(d,2N\square)=1\\ 0<8d\leq X}}\left| L\left(\frac{1}{2}+\frac{1}{\log X\kappa N} +it, f \otimes \chi_{8d}\right) \right| \,dt \ll X \left(\log X\kappa N\right)^{1/4+\varepsilon}.\] We also have the estimate \[ \int_{-\infty}^\infty \frac{\left| G\left(\frac{1}{\log X\kappa N}+it\right)\right|^\frac{1}{2}}{\left|\frac{1}{\log X\kappa N}+it\right|}\,dt \ll \log \log X\kappa N,\] so that pulling these estimates together we obtain \[ \sum_{(d,2N\square)=1}\mathcal{B}_U(1/2,f\otimes \chi_{8d})F(8d/X) \ll_{\varepsilon,A} X\left(\log X \kappa N\right)^{1/4+\varepsilon} \] hence the theorem.  \end{proof}

\section{Proof of Proposition \ref{FirstMomentCalculation}}\label{pfprop3}

We treat the two cases $N'=1$ and $N'=N$ somewhat differently.  In the case $N'=1$, the dependence on $N$ in $T(1,h)$ appears only in the relatively prime condition, which we may we treat solely by M\"{o}buis inversion.  In the case of $T(N,h)$, the dependence on $N$ is carried through the average over quadratic characters, but there is one less inversion to preform, making the calculation a bit simpler.

\begin{proof}
The condition $(d,2N\square)=1$ has been introduced to the sum over twists to restrict $8d$ to lie in a large subsequence of fundamental discriminants.  However, this condition is awkward to work with and our first task will be to remove it.

\subsection{Preliminary simplifications, $N'=1$ case}

We start with the $N'=1$ case.  We use M\"{o}bius inversion to remove both the squarefree and relatively prime to $N$ conditions from the sum over $d$, \[T(1,h) = \sum_{(a_1,2N)=1}\mu(a_1) \sum_{a_2|N} \mu(a_2) \sum_{(d,2)=1} \sum_{(n,a)=1} \frac{\lambda_f(n)}{n^{1/2}} \chi_{8da_2}(n) h(da_1^2a_2,n).\] Split the sums over $a_1$ and $a_2$ at $Y_1$ and $Y_2$ in to tail and main term.  This splitting results in $4$ truncated sums: \[T(1,h) = T_1(1,h) + T_{21}(1,h)+ T_{22}(1,h)+T_{23}(1,h)\] where we have defined \[T_1(1,h) := \sum_{\substack{(a_1,2N)=1 \\ a_1 \leq Y_1}}\mu(a_1) \sum_{\substack{a_2|N \\ a_2 \leq Y_2}} \mu(a_2) \sum_{(d,2)=1} \sum_{(n,a_1)=1} \frac{\lambda_f(n)}{n^{1/2}} \chi_{8da_2}(n) h(da_1^2a_2,n),\]  \[T_{21}(1,h) := \sum_{\substack{(a_1,2N)=1 \\ a_1 \leq Y_1}}\mu(a_1) \sum_{\substack{a_2|N \\ a_2 > Y_2}} \mu(a_2) \sum_{(d,2)=1} \sum_{(n,a_1)=1} \frac{\lambda_f(n)}{n^{1/2}} \chi_{8da_2}(n) h(da_1^2a_2,n),\]  \[T_{22}(1,h) := \sum_{\substack{(a_1,2N)=1 \\  a_1 > Y_1}}\mu(a_1) \sum_{\substack{a_2|N \\ a_2 \leq Y_2}} \mu(a_2) \sum_{(d,2)=1} \sum_{(n,a_1)=1} \frac{\lambda_f(n)}{n^{1/2}} \chi_{8da_2}(n) h(da_1^2a_2,n),\] \[T_{23}(1,h) := \sum_{\substack{(a_1,2N)=1 \\  a_1 > Y_1}}\mu(a_1) \sum_{\substack{a_2|N \\ a_2 > Y_2}} \mu(a_2) \sum_{(d,2)=1} \sum_{(n,a_1)=1} \frac{\lambda_f(n)}{n^{1/2}} \chi_{8da_2}(n) h(da_1^2a_2,n).\]  The main term will come from the most difficult sum $T_1(1,h)$.  First, however, we estimate the other cases $T_{21}(1,h)$, $T_{22}(1,h)$ and $T_{23}(1,h)$.   

\begin{lemma}\label{T_2} Assume GRH or conjecture \ref{conjecture2}.  We have the bounds \[T_{21}(1,h) \ll \frac{X}{Y_2} (\log X \kappa N)^5, \] \[T_{22}(1,h) \ll \frac{X}{Y_1} (\log X \kappa N)^5, \] \[T_{23}(1,h) \ll \frac{X}{Y_1Y_2} (\log X \kappa N)^5.\]  \end{lemma} \begin{proof} Consider the case $T_{21}(1,h)$ and write $d= b_1^2 b_2 \ell$ with $(\ell,2N\square)=1,$ $(b_1,N)=1$ and $b_2 \mid N$.  Group the variables as $c_1 = a_1b_1$ and $c_2 = a_2b_2$ to obtain \[T_{21}(1,h)=\sum_{(c_1,2N)=1} \sum_{c_2 | N} \sum_{\substack{a_1 \mid c_1 \\ a_1 \leq Y_1}} \mu(a_1) \sum_{\substack{a_2 \mid c_2 \\ a_2 > Y_2}} \mu(a_2) \sum_{(\ell,2N\square)=1} \sum_{(n,c_1)=1} \frac{\lambda_f(n)}{n^{1/2}} \chi_{8\ell c_2}(n) h(\ell c_1^2c_2,n).\]  Let $f_{c_2}$ denote the newform given by the quadratic twist $f\otimes \chi_{c_2},$ which is of some level dividing $N^2$.   Set $\check{h}(x,u) :=\int_0^\infty h(x,y)y^{u-1}\,dy,$ which by repeated partial integration can be estimated by \[\check{h}(x,u) \ll \left(1+\frac{x}{X}\right)^{-100} \frac{(U\kappa \sqrt{N}z)^{\real(u)} }{|u|^2(1+|u|)^{10}}.\]  We then have by Mellin inversion that $T_{21}(1,h)$ is \[=     \sum_{(c_1,2N)=1} \sum_{c_2 | N} \sum_{\substack{a_1 \mid c_1 \\ a_1 \leq Y_1}} \mu(a_1) \sum_{\substack{a_2 \mid c_2 \\ a_2 > Y_2}} \mu(a_2) \frac{1}{2\pi i } \int_{(1/2+\varepsilon)} \sum_{(\ell,2 N \square)=1} \check{h}(\ell c_1^2c_2,u) L_{c_1}(1/2+u,f_{c_2}\otimes \chi_{8\ell})\,du,\] where $L_{c_1}(1/2+u,f_{c_2}\otimes \chi_{8\ell})$ is the function formed from the same Euler product as $L(1/2+u,f_{c_2}\otimes \chi_{8\ell}),$ but with those factors at primes dividing $c_1$ omitted.  We have that \[|L_{c_1}(1/2+u,f_{c_2}\otimes \chi_{8\ell})| \leq d(c_1) | L(1/2+u,f_{c_2}\otimes \chi_{8\ell}) | ,\] so that shifting the contour to the line $\real(u) = 1/\log X\kappa N,$ we have that $T_{21}(1,h)$ is \begin{equation*}\begin{split} \ll (\log X\kappa N)^2 \sum_{(c_1,2N)=1} d(c_1)\sum_{c_2 | N} \sum_{\substack{a_1 \mid c_1 \\ a_1 \leq Y_1}} \sum_{\substack{a_2 \mid c_2 \\ a_2 > Y_2}} \int_{-\infty}^\infty \sum_{(\ell,2N\square)=1} \left(1+\frac{\ell c_1^2c_2}{X}\right)^{-100} \\ \times \frac{|L(1/2+1/\log X\kappa N + it,f_{c_2}\otimes \chi_{8\ell}) |}{(1+|t|)^{10}}\,dt.\end{split} \end{equation*}  Using conjecture \ref{conjecture2} (i.e. GRH) we find that \begin{eqnarray*} T_{21}(1,h) &\ll &  X (\log X\kappa N)^3 \sum_{(c_1,2N)=1} \frac{d(c_1)}{c_1^2}\sum_{c_2 | N}\frac{1}{c_2} \sum_{\substack{a_1 \mid c_1 \\ a_1 \leq Y_1}} \sum_{\substack{a_2 \mid c_2 \\ a_2 > Y_2}}  \\ &\ll &  X (\log X\kappa N)^5. \end{eqnarray*} The cases $T_{22}(1,h)$ and $T_{23}(1,h)$ are treated similarly.  \end{proof}

\subsection{Averaging quadratic characters}
We now turn to $T_1(1,h).$  We quote two very useful lemmas from \cite{Sound2000}.  The first is lemma 2.6 of \cite{Sound2000}, which is the trace formula for quadratic characters.  \begin{lemma}[Poisson Summation]\label{PoissonSummation} Let $F$ be a smooth function with compact support on the positive real numbers, and suppose that $n$ is an odd integer.  Then \[\sum_{(d,2)=1} \left(\frac{d}{n}\right) F\left(\frac{d}{Z}\right) = \frac{Z}{2n} \left(\frac{2}{n}\right)\sum_{k\in \Z} (-1)^k G_k(n)\widehat{F}\left(\frac{kZ}{2n}\right),\] where \[G_k(n) = \left(\frac{1-i}{2}+\left(\frac{-1}{n}\right)\frac{1+i}{2}\right) \sum_{a \pmod n} \left(\frac{a}{n}\right)e\left(\frac{ak}{n}\right),\] and \[\widehat{F}(y) = \int_{-\infty}^\infty \left(\cos(2 \pi x y )+\sin(2 \pi x y)\right)F(x)\,dx\] is a Fourier-type transform of $F$.  \end{lemma}  The Gauss-type sum $G_k(n)$ has the following explicit evaluation from lemma 2.3 of \cite{Sound2000}:   \begin{lemma}\label{GaussSum} If $m$ and $n$ are relatively prime odd integers, then $G_k(mn) = G_k(m)G_k(n)$, and if $p^\alpha$ is the largest power of $p$ dividing $k$ (setting $\alpha = \infty$ if $k=0$), then \[ G_k(p^\beta)= \begin{cases} 0 & \text{ if } \beta \leq \alpha \text{ is odd} \\ \phi(p^\beta) & \text{ if } \beta \leq \alpha \text{ is even } \\ -p^\alpha & \text{ if } \beta = \alpha + 1\text{ is even } \\ \left( \frac{kp^{-\alpha}}{p}\right) p^\alpha \sqrt{p} & \text{ if } \beta = \alpha +1 \text{ is odd} \\ 0 & \text{ if } \beta \geq \alpha + 2. \end{cases} \] \end{lemma} Applying these lemmas to $T_1(1,h)$ we find that \begin{equation}\label{like305}\begin{split} T_1(1,h) = \frac{X}{2}\sum_{\substack{(a_1,2N)=1 \\ a_1 \leq Y_1}}\frac{\mu(a_1)}{a_1^2} \sum_{\substack{a_2|N\\ a_2 \leq Y_2}} \frac{\mu(a_2)}{a_2}  \sum_{k \in \Z} (-1)^k \sum_{(n,2a_1)=1} \frac{\lambda_f(n)}{n^{1/2}} \chi_{a_2}(n)  \frac{G_k(n)}{n} \\ \times \int_0^\infty \left(\sin+\cos \right) \left(\frac{2\pi k xX}{2 n a_1^2 a_2}\right)h(xX,n)\,dx. \end{split}\end{equation}  

\subsection{The main term}\label{mainterm}
The main term of $T_1(1,h)$ is from the $k=0$ term of \eqref{like305}, which we extract and analyze.  Call the $k=0$ term $T_{10}(1,h),$ and observe from lemma \ref{GaussSum} that $G_0(n)\neq 0$ if and only if $n$ is a square, in which case $G_0(n) = \phi(n).$  Setting $h_1(n) = \int_0^\infty h(xX,n)\,dx$, we find \begin{eqnarray*} T_{10}(1,h) & = &  \frac{X}{2} \sum_{\substack{(a_1,2N)=1 \\ a_1 \leq Y_1}} \frac{\mu(a_1)}{a_1^2} \sum_{\substack{a_2 \mid N  \\ a_2 \leq Y_2}} \frac{\mu(a_2)}{a_2} \sum_{\substack{(n,2a_1a_2)=1 \\ n=\square}} \frac{\lambda_f(n)}{n^{1/2}} \prod_{p\mid n}\left(1-\frac{1}{p}\right) h_1(n) \\ &  = &  \frac{2X}{3\zeta(2)} \sum_{\substack{(n,2)=1 \\ n = \square} } \frac{\lambda_f(n)}{n^{1/2}} \prod_{p \mid nN}\frac{p}{p+1} h_1(n) + O\left( X\left(\frac{1}{Y_1}+\frac{1}{Y_2}\right)\sum_{\substack{(n,2)=1 \\ n=\square}} \frac{d(n)}{n^{1/2}} |h_1(n)| \right), \end{eqnarray*} so that using the bounds on $h$ in the statement of the proposition we have \begin{equation}\label{MT} T_{10}(1,h) = \frac{4X}{\pi^2}  \sum_{\substack{(n,2)=1 \\ n = \square} } \frac{\lambda_f(n)}{n^{1/2}} \prod_{p \mid nN}\frac{p}{p+1} h_1(n) + O\left( X\left(\frac{1}{Y_1}+\frac{1}{Y_2}\right)(\log X\kappa N)^4\right).\end{equation}

\subsection{Bounding the dual sum}
We now proceed to the $k \neq 0$ terms of $T_1(1,h)$, which we call $T_3(1,h)$.  Our first task is to express the integral in \eqref{like305} in terms of Mellin inverses.  \begin{lemma}\label{transform} Let $k\neq 0,$ $X>1$ and let $h(x,y)$ be as in the statement of the theorem.  Define the transform \[\widetilde{h}(s,u) := \int_0^\infty \int_0^\infty h(x,y)x^sy^u\frac{ds}{s}\,\frac{du}{u}.\]  Then we have \begin{equation}\label{inversion}\begin{split} \int_0^\infty \left(\sin+\cos\right) \left(\frac{2\pi k xX}{2 n a_1^2 a_2}\right)h(xX,n)\,dx = \frac{1}{X}\frac{1}{(2 \pi i )^2} \int_{(\varepsilon)}\int_{(\varepsilon)} \widetilde{h}(1-s,u)\frac{1}{n^{u}}\left(\frac{na_1^2a_2}{\pi|k|}\right)^s \\ \times \Gamma(s) \left(\cos+ \sgn(k)\sin\right)\left(\frac{\pi s}{2}\right)\,ds\,du.\end{split}\end{equation}  Moreover one has the bounds \[\widetilde{h}(s,u) \ll \frac{ (U\kappa \sqrt{N})^{\real(u)}X^{\real(s)} }{|u|^2(1+|u|)^{98}(1+|s|)^{98}}.\]  \end{lemma}

\begin{proof} Use the formulae for the Mellin transforms of $\sin$ and $\cos$ and Mellin inversion.  See Soundararajan and Young \cite{SoundYoung}, Section 3.3.  \end{proof}
 
Inspecting lemma \ref{GaussSum}, we find that for odd $n$, $G_k(n) = G_{4k}(n),$ so that inserting the formula of lemma \ref{transform} in \eqref{like305}, one finds that \begin{equation*}\begin{split}T_3(1,h) = \frac{1}{2} \sum_{\substack{(a_1,2N)=1 \\ a_1 \leq Y_1}}\frac{\mu(a_1)}{a_1^2} \sum_{\substack{a_2|N\\ a_2 \leq Y_2}} \frac{\mu(a_2)}{a_2}  \sum_{k \in \Z} (-1)^k \sum_{(n,2a_1)=1} \frac{\lambda_f(n)}{n^{1/2}} \chi_{a_2}(n)  \frac{G_{4k}(n)}{n}\frac{1}{(2 \pi i )^2} \int_{(\varepsilon)}\int_{(\varepsilon)} \widetilde{h}(1-s,u) \\ \times \frac{1}{n^{u}}\left(\frac{na_1^2a_2}{\pi|k|}\right)^s \Gamma(s) \left(\cos+ \sgn(k)\sin \right)\left(\frac{\pi s}{2}\right)\,ds\,du.\end{split}\end{equation*} Recall that we denote the set of fundamental discriminants by $\mathcal{D}$.  Now set $4k= k_1k_2^2k_3,$ where $k_1k_3 \in \mathcal{D}$, and $(k_1,N)=1$ but $k_3 \mid N$.  Define the function \begin{equation}\label{DirichletSeriesZ} Z_1(\alpha, \gamma,q_1,q_2,k_1k_3) := \sum_{k_2=1}^\infty \sum_{(n,2q_1)=1} \frac{\lambda_f(n)\chi_{q_2}(n)}{n^\alpha|k_2|^{2\gamma}} \frac{G_{k_1k_2^2k_3}(n)}{n},\end{equation} and set \begin{equation}\label{H} H(s):= \Gamma(s)\left(\cos+\sgn(k) \sin\right)\left(\frac{\pi s}{2}\right)\left(1-2^{1-2s}\right)^{-1}\ll |s|^{\real(s)-1/2}.\end{equation}   Splitting up $4k$ in this manner and after a change of variables one finds that \begin{equation}\label{like311}\begin{split}T_3(1,h) = \frac{1}{2} \sum_{\substack{(a_1,2N)=1 \\ a_1 \leq Y_1}}\frac{\mu(a_1)}{a_1^2} \sum_{\substack{a_2|N\\ a_2 \leq Y_2}} \frac{\mu(a_2)}{a_2} \sum_{\substack{k_3 \in \mathcal{D} \\ k_3 | N}} \sum_{\substack{k_1\in \mathcal{D} \\ (k_1,N)=1}} (-1)^{k_1k_3}  \frac{1}{(2\pi i )^2} \int_{(1/2+\varepsilon)} \int_{(\varepsilon)} \widetilde{h}(1-s,u+s) \\ \times \left(\frac{a_1^2a_2}{\pi | k_1k_3|}\right)^s H(s) Z_1(1/2+u,s,a_1,a_2,k_1k_3)\,du\,ds.\end{split} \end{equation}  To estimate $T_3(1,h)$ by contour shifting, we must analyze the Dirichlet series $Z_1$.  \begin{lemma}\label{likeLemma331} Let $k_1k_3$ be a fundamental discriminant, where $k_3 \mid N$ but $(k_1,N)=1$, and $q_1,q_2$ positive integers where $q_2 \mid N$ and $(q_1,2N)=1$.  Denote by $f_{k_3q_2}$ the newform defined by the quadratic twist $f \otimes \chi_{k_3q_2}$, which is of some level dividing $N^3$.  For the Dirichlet series defined by \eqref{DirichletSeriesZ} one has \[Z_1(\alpha, \gamma,q_1,q_2,k_1k_3)=\frac{L_{q_1q_2}(1/2+\alpha,f_{k_3q_2}\otimes \chi_{k_1})}{L_{q_1q_2}(1+2\alpha,\sym^2 f)}Z_1^*(\alpha,\gamma,q_1,q_2,k_1k_3),\] where subscripts denote the omission of Euler factors, and $Z_1^*$ is given by some Euler product absolutely convergent in $\real(\alpha)\geq 0$ and $\real(\gamma) \geq 1/2+\varepsilon$ and uniformly bounded in $q_1,q_2,k_1,k_3,\kappa$ and $N$.  \end{lemma} \begin{proof} By lemma \ref{GaussSum}, the terms of the Dirichlet series defining $Z$ are joint multiplicative in $n$ and $k_2$, so that we may decompose $Z$ as an Euler product.  The generic Euler factor is given by \[ \sum_{k_2,n \geq 0} \frac{\lambda_f(p^n)\chi_{q_2}(p)^n}{p^{n\alpha + 2\gamma k_2}}\frac{G_{k_1k_3p^{2k_2}}(p^n)}{p^n},\] and we must check the several cases where $p$ divides the various parameters $N,q_1,q_2,k_1,k_3,$ or not.  First, we consider the generic case where $p \nmid 2Nq_1q_2k_1k_3.$  By lemma \ref{GaussSum}, we find that the terms $k_2\geq 1$ contribute $\ll p^{-(1+2\varepsilon)},$ and the $k_2 = 0$ terms are exactly \[1+\frac{\lambda_f(p)\chi_{k_3q_2}(p)\chi_{k_1}(p)}{p^{1/2+\alpha}},\] so that these Euler factors match those in the statement of the lemma.  Next, consider the cases $p \mid k_1,$ $p\nmid 2Nq_1q_2k_3,$ or $p \mid k_3$, $ p \nmid 2q_1q_2k_1$.  In either of these two cases we check that such an Euler factor is \[1-\frac{\lambda_f(p^2)}{p^{1+2\alpha}} + O(p^{-(1+2\varepsilon)}),\] which again matches the Euler factor in the lemma.   If $p \mid N$, but $p \nmid 2q_1q_2k_1k_3$, then this Euler factor is \[1+\frac{\lambda_f(p)\chi_{k_3q_2}(p)\chi_{k_1}(p)}{p^{1/2+\alpha}} + O(p^{-(1+2\varepsilon)}).\]  Observing that $\lambda_f(p^2)=\lambda_f(p)^2$ for primes dividing the level, the also matches the Euler factor from the statement of the lemma.  Finally, if $p\mid q_1q_2,$ then all terms $n\geq 1$ vanish, and the contribution of such an Euler factor is $1+O(p^{-(1+2\varepsilon)})$. \end{proof} 

We now return to \eqref{like311}, and split the sum over $k_1$ at $U\kappa \sqrt{N}z Y_1^2Y_2^{1/2}/X$.  For the small $k_1$ terms, we shift the lines of integration to $\real(u) = -1/2+ 1/\log X \kappa N$ and $\real(s) = 3/4$, and for the large $k_1$ terms to $\real(u) = -1/2+1/\log X\kappa N$ and $\real(s) = 5/4$.  Recall that $H(s) \ll |s|^{\real(s)-1/2},$ and observe \[|L_{a_1a_2}(1/2+\alpha,f_{k_3q_2} \otimes \chi_{k_1})| \leq d(a_1)d(a_2) |L(1/2+\alpha,f_{k_3q_2} \otimes \chi_{k_1})|.\] Applying the result of Goldfeld, Hoffstein and Lieman \cite{HLAppendix}, the small $k_1$ terms are \begin{equation*}\begin{split} \ll (\log X\kappa N)^3 (U \kappa \sqrt{N}z)^{1/4} X^{1/4} \sum_{\substack{(a_1,2N)=1 \\ a_1 \leq Y_1}}\frac{d(a_1)}{\sqrt{a_1}} \sum_{\substack{a_2|N\\ a_2 \leq Y_2}} \frac{d(a_2)}{a_2^{1/4}} \sum_{\substack{k_3 \in \mathcal{D} \\ k_3 | N}} \frac{1}{|k_3|^{3/4}} \int_{(3/4)}\int_{(-1/2+1/\log X\kappa N)} \\ \times \sum_{\substack{|k_1|\leq  U\kappa \sqrt{N} zY_1^2Y_2^{1/2}/X \\ k_1\in \mathcal{D} \\ (k_1,N)=1  }} \frac{|L(1+u,f_{k_3a_2} \otimes \chi_{k_1})|}{|k_1|^{3/4}}\frac{ds \,du}{(1+|s|)^{98}(1+|u|)^{98}}.\end{split}\end{equation*}  Using conjecture \ref{conjecture2}, i.e. GRH, we find that this is $\ll (U\kappa \sqrt{N}z)^{1/2}  Y_1 Y_2^{1/8} (\log X \kappa N)^6.$  Now consider the large $k_1$ terms.  Similarly, their contribution is \begin{equation*}\begin{split} \ll (\log X\kappa N)^3 \frac{(U \kappa \sqrt{N}z)^{3/4}}{ X^{1/4}} \sum_{\substack{(a_1,2N)=1 \\ a_1 \leq Y_1}}d(a_1)\sqrt{a_1} \sum_{\substack{a_2|N\\ a_2 \leq Y_2}} d(a_2)a_2^{1/4} \sum_{\substack{k_3 \in \mathcal{D} \\ k_3 | N}} \frac{1}{|k_3|^{5/4}} \int_{(5/4)}\int_{(-1/2+1/\log X\kappa N)} \\ \times \sum_{\substack{|k_1|>  U\kappa \sqrt{N}z Y_1^2Y_2^{1/2}/X \\ k_1\in \mathcal{D} \\ (k_1,N)=1  }} \frac{|L(1+u,f_{k_3a_2} \otimes \chi_{k_1})|}{|k_1|^{5/4}}\frac{ds \,du}{(1+|s|)^{98}(1+|u|)^{98}}.\end{split}\end{equation*} Again, by conjecture \ref{conjecture2}, this is $\ll (U\kappa \sqrt{N}z)^{1/2} Y_1 Y_2^{1/8} (\log X \kappa N)^6.$   Taking\[ Y_1=Y_2= \frac{X^{8/17}}{(U\kappa \sqrt{N}z)^{4/17}}, \]  we find that \[T_3(1,h) \ll X^{9/17}(U\kappa \sqrt{N}z)^{4/17}\left(\log X \kappa N \right)^6,\] and drawing all error terms together we obtain the proposition for $N'=1$.  

\subsection{The $N'=N$ case}

The proof in the $T(N,h)$ case follows the same outline as in the $T(1,h)$ case, above.  We need only M\"{o}bius invert the squarefree condition and not the relatively prime to $N$ condition, but we must keep careful track of the dependence on $N$ in the analogue of lemma \ref{likeLemma331}.  We sketch the argument, omitting those details which are similar to those of $T(1,h)$.  

We begin by using M\"{o}bius inversion to remove the squarefree condition and split the resulting sum at $Y$.  \begin{eqnarray*} T(N,h)& = & \left(\sum_{\substack{a \leq Y \\ (a,2N)=1}} + \sum_{\substack{a > Y \\ (a,2N)=1}} \right) \mu(a) \sum_{(d,2N)=1} \sum_{(n,a)=1} \frac{\lambda_f(n)}{\sqrt{n}} \chi_{8d}(Nn)h(da^2,n) \\& =: & T_1(N,h) + T_2(N,h). \end{eqnarray*}  By a slight modification of lemma \ref{T_2}, we find that \[T_2(N,h) \ll \frac{X}{Y} (\log X\kappa N)^5, \] and so we concentrate on $T_1(N,h)$.  Applying Poisson summation (lemma \ref{PoissonSummation}), we have that  \[T_1(N,h) = \frac{X}{2} \sum_{\substack{a \leq Y \\ (a, 2N)=1}} \frac{\mu(a)}{a^2}  \sum_{k \in \Z} (-1)^k \sum_{(n,2a)=1} \frac{\lambda_f(n)}{\sqrt{n}}\frac{G_k( N n)}{Nn} \int_0^\infty \left(\cos + \sin \right)\left(\frac{2\pi k xX}{2Nna^2 }\right)h(xX,n)\,dx.\]  Now we pick out from $T_1(N',h)$ the main term, which is when $k=0$, and call it $T_{10}(N,h)$.  By pulling the sum over $a$ inside and computing as in Subsection \ref{mainterm}, we find that \[T_{10}(N,h) = \frac{4X}{\pi^2} \sum_{\substack{(n,2)=1 \\ Nn=\square}} \frac{\lambda_f(n)}{\sqrt{n}} \prod_{p|Nn} \frac{p}{p+1}h_1(n) + O\left( \frac{X}{Y} \left(\log X\kappa N\right)^3  \right).\]  

Now we turn to the $k \neq 0$ terms of $T_1(N,h)$ and call them $T_3(N,h).$  Define \begin{equation}\label{DirichletSeriesZN} Z_N(\alpha, \gamma, q, k_1k_3) := \sum_{k_2=1}^\infty \sum_{(n,2q)=1} \frac{\lambda_f(n)}{n^\alpha}\left(\frac{N}{|k_2|^2}\right)^\gamma \frac{G_{k_1k_2^2k_3}(Nn)}{Nn}.\end{equation} Recall the definition of $H(s)$ from \eqref{H} and apply the inversion formula \eqref{inversion} for the weight function, to find the analogue of formula \eqref{like311}: \begin{equation}\label{like311again}\begin{split} T_{3}(N,h)= \frac{1}{2} \sum_{\substack{a \leq Y \\ (a, 2N)=1}} \frac{\mu(a)}{a^2} \sum_{\substack{k_3\in\mathcal{D} \\ k_3 \mid N}}\sum_{\substack{k_1 \in \mathcal{D} \\ (k_1,N)=1}} (-1)^{k_1k_3}  \frac{1}{(2 \pi i)^2} \int_{(\varepsilon)}\int_{(1/2+\varepsilon)} \widetilde{h}(1-s,u+s) \left(\frac{a^2}{\pi |k_1k_3|}\right)^s \\ \times H(s) Z_{N}(1/2+u,s,a,k_1k_3) \,ds\,du.\end{split}\end{equation}  In order to use contour shifting, we analyze the Dirichlet series $Z_N$, taking special care with the dependence on $N$.  \begin{lemma}\label{DirichletSeries} Let $k_1k_3$ be a fundamental discriminant, where $k_3 \mid N$ but $(k_1,N)=1$, and $q$ positive integer relatively prime to $2N$.  Denote by $f_{k_3}$ the newform defined by the quadratic twist $f \otimes \chi_{k_3}$, which is of some level dividing $N^2$.  For the Dirichlet series defined by \eqref{DirichletSeriesZN} one has \[Z_{N}(\alpha, \gamma, q,k_1k_3)=\frac{L_{qN}(1/2+\alpha, f_{k_3}\otimes \chi_{k_1})}{L_{qN}(1+2\alpha,\sym^2 f)}Z_{N}^*(\alpha,\gamma, q,k_1k_3),\] where $Z_{N}^*\ll d(N) N^{\real(\gamma)-1/2},$ uniformly in $q, \kappa, k_1,k_3,\real(\gamma) >1/2+\varepsilon$, $\real(\alpha) \geq 0.$  \end{lemma} \begin{proof}  From lemma \ref{GaussSum} we see that the summand is within a constant of being jointly multiplicative in $n,k_2$, so that we may write an Euler product.  We use the notation $p^r || N$ to mean that $r$ is the largest power of $p$ dividing $N$.  Then $Z_{N}$ is given by \begin{equation}\label{prod}\prod_{p \nmid 2N} \sum_{k_2,n \geq 0} \frac{\lambda_f(p^n)}{p^{n\alpha+2\gamma k_2}} \frac{G_{k_1k_3p^{2k_2}}(p^n)}{p^n} \prod_{p^r || N}\sum_{k_2,n \geq 0} \frac{\lambda_f(p^n)}{p^{n\alpha + 2 \gamma k_2-r\gamma}} \frac{G_{k_1k_3p^{2k_2}}(p^{r+n})}{p^{r+n}}.\end{equation} We must check all possible cases when $p$ does or does not divide the parameters $N,q,k_1$ and $k_3$.  Let us begin with the generic $p\nmid N$.  Suppose first that $p\nmid 2qk_1$.  We have that all of the terms where $k_2\geq 1$ contribute $\ll p^{-(1+2\varepsilon)},$ uniformly in all parameters.  The $k_2=0$ terms are exactly \[1+ \lambda_f(p) \chi_{k_1k_3}(p) p^{-(1/2+\alpha)},\] which matches the proposed Euler factor in the statement of the lemma up to a uniformly bounded factor.  Now we consider the terms with $p \nmid 2q$ but $p \mid k_1$.  In this case, the Euler factor is given by \[1-\frac{\lambda_f(p^2)}{p^{1+2\alpha}} + O\left( \frac{1}{p^{1+2\varepsilon}}\right),\] which exactly matches the Euler factor in the statement of the lemma up to a uniformly bounded factor in $N,\kappa, k_1$.  If $p \mid 2q$, then the Euler factor is $1+O(p^{-(1+2\varepsilon)}).$  

Now we turn to the terms where $p \mid N$.  Inspecting lemma \ref{GaussSum} we find four cases depending on whether $p$ divides $N$ to even or odd order and whether $p \mid k_3$ or not.   When $r$ is odd the second product of \eqref{prod} is \[\prod_{\substack{ p^r || N \\ r \text{ odd} \\ p \nmid k_3}} \left( \chi_{k_1k_3}(p) p^{\gamma-1/2} + \frac{\lambda_f(p)}{p^{\alpha+\gamma}}+O(p^{-(1+\varepsilon)})\right)\prod_{\substack{p^r || N \\ r \text{ odd} \\ p\mid k_3}} \left(- \frac{\lambda_f(p) }{p^{ \alpha}}p^{\gamma -1}+\frac{\lambda_f(p)}{p^{\alpha+\gamma}}+ O(p^{-(1+2\varepsilon)})\right),\] and when $r$ is even this product is \[\prod_{\substack{p^r || N \\ r \text{ even} \\ p \nmid k_3}} \left(1+\frac{\lambda_f(p)\chi_{k_1k_3}(p)}{p^{\alpha+1/2}}+O(p^{-(1+2\varepsilon)})\right) \prod_{\substack{ p^r || N \\ r \text{ even} \\ p \mid k_3}} \left( -p^{2\gamma -1} + 1 -\frac{1}{p} -\frac{\lambda_f(p^2)}{p^{2\alpha+1}}+ O(p^{-(1+2\varepsilon)})\right).\]  If $r$ is even then $r\geq 2$, so we have that $Z_{N}^* \ll d(N)N^{\real(\gamma)-1/2}.$ 
\end{proof}

Now we return to $T_3(N,h)$, and split the sum over $k_1$ at $U\kappa N^{3/2}zY^2/X$.  When $|k_1| \leq U\kappa N^{3/2}zY^2/X$, shift the lines of integration to $\real(s) = 3/4$ and $\real(u) = -1/2 + 1/\log X\kappa N$, and for the tail $k_1$, shift to $\real(s) = 5/4$ and $\real(u) = -1/2+1/\log X\kappa N$.  We have that \begin{eqnarray*} Z_{N}(1/2+u,s,a,k_1) & \ll & |L_{aN}(1+u,f_{k_3}\otimes \chi_{k_1})| (\log X\kappa N)^2 N^{\real(\gamma)-1/2} \\ & \ll & (\log X\kappa N)^3 \prod_{p|a}\left(1+\frac{10}{\sqrt{p}}\right) |L(1+u,f_{k_3} \otimes \chi_{k_1})| N^{\real(\gamma)-1/2}\end{eqnarray*} unconditionally due to the work of Goldfeld, Hoffstein and Lieman \cite{HLAppendix}.  We also have the estimate $H(s)\ll |s|^{\real(s)-1/2},$ so that the small $k_1$ of $T_3(N,h)$ are \begin{equation*}\label{313}\begin{split} \ll (XU \kappa \sqrt{N}z)^{1/4}N^{1/4}(\log X\kappa N)^5 \sum_{a\leq Y} \frac{1}{\sqrt{a}} \prod_{p|a}\left(1+\frac{10}{\sqrt{p}}\right)  \sum_{\substack{k_3 \in \mathcal{D} \\ k_3 | N}} \frac{1}{|k_3|^{3/4}}\int_{(3/4)} \int_{(-1/2+1/\log X\kappa N)}\\ \times \sum_{\substack{|k_1| \leq U\kappa N^{3/2}zY^2/X \\ k_1\in \mathcal{D} \\ (k_1,N)=1}} \frac{1}{|k_1|^{3/4}} |L(1+u,f_{k_3}\otimes \chi_{k_1})| \frac{du\,ds}{(1+|s|)^{98}(1+|u|)^{98}}.   \end{split}\end{equation*}  We have that by conjecture \ref{conjecture2} this is $\ll (U\kappa N^{3/2}z)^{1/2}Y (\log X\kappa N)^6$.  Similarly the tail $k_1$ terms are \begin{equation*}\label{314}\begin{split}\ll (U\kappa \sqrt{N}z)^{3/4} X^{-1/4} N^{3/4}(\log X\kappa N)^5 \sum_{a \leq Y}\sqrt{a} \prod_{p|a}\left(1+\frac{10}{\sqrt{p}}\right)  \sum_{\substack{k_3 \in \mathcal{D} \\ k_3 | N}} \frac{1}{|k_3|^{5/4}}  \int_{(5/4)} \int_{(-1/2+1/\log X\kappa N)}\\ \times \sum_{\substack{|k_1| > U\kappa N^{3/2}zY^2/X \\ k_1\in \mathcal{D} \\ (k_1,N)=1 }} \frac{1}{|k_1|^{5/4}} |L(1+u,f_{k_3}\otimes \chi_{k_1})| \frac{du\,ds}{(1+|s|)^{98}(1+|u|)^{98}},   \end{split}\end{equation*} which is $\ll  (U\kappa N^{3/2}z)^{1/2}Y (\log X\kappa N)^6$ as well by conjecture \ref{conjecture2}.  Taking $Y=X^{1/2}/(U\kappa N^{3/2})^{1/4},$ we find \[T_3(N,h) \ll X^{1/2} (U \kappa N^{3/2}z)^{1/4} (\log X\kappa N)^6.\] 

\end{proof}

\end{document}